\documentclass{article}
\usepackage{hyperref}
\usepackage{enumerate}
\usepackage{cite}
\usepackage{amsfonts}
\usepackage{amsthm}
\usepackage{amsmath,mathtools}
\usepackage{amssymb}
\usepackage{caption}
\usepackage{comment}
\usepackage{subcaption}
\usepackage{xcolor}
\usepackage{soul} 
\usepackage{subcaption}
\usepackage{graphicx}
\usepackage{tikz}
\usepackage{pgfplots}
\usepackage{pgfplotstable}
\usepackage{tkz-fct}
\usetikzlibrary{pgfplots.polar}
\pgfplotsset{compat=newest}
\usepgfplotslibrary{fillbetween}
\usetikzlibrary{patterns}
\usepackage{mathrsfs}
\usepackage{authblk}
\usepackage[utf8]{inputenc}
\usepackage{array}
\usepackage{pgffor}
\usepackage{float}

\usepackage{multirow}
\usepackage[top=1in, bottom=1in, left=1in, right=1in]{geometry}
\hypersetup{colorlinks, linkcolor=red, filecolor=darkgreen, urlcolor=blue, citecolor=blue}
\theoremstyle{plain}
\newtheorem{theorem}{Theorem}[section]

\theoremstyle{definition}
\newtheorem{definition}{Definition}[section]

\numberwithin{equation}{section}

\definecolor{pomcol}{rgb}{1,0,0}
\definecolor{dbcol}{rgb}{0,0,0.8}

\title{Fractal-Fractional HIV Dynamics with Mittag-Leffler Kernel: Analysis, Stability, and Numerical Simulations}
 \author{Niaz Ali Shah$^{1}$, Samad Noeiaghdam$^{2,3,*}$\\
 \small $^{1}$Department of Mathematics Abbottabad University of Science and Technology, Abbottabad, 22010 Pakistan. nuzhatniaz2014@gmail.com  \\
 \small $^{2}$Institute of Mathematics, Henan Academy of Sciences, Zhengzhou 450046, China. snoei@hnas.ac.cn\\
 \small $^{3}$Department of Mathematical Sciences, Saveetha School of Engineering, SIMATS, Chennai, 602105, India.\\
}
\date{}
\begin{document}
\maketitle

\begin{abstract}

In this paper, a fractal--fractional HIV model with the Mittag--Leffler kernel is proposed using the Atangana--Baleanu--Caputo operator to capture the memory and hereditary properties of the disease dynamics. The existence and uniqueness of the solutions are investigated using suitable analytical techniques, and the Hyers--Ulam stability analysis is carried out to verify the stability behavior of the proposed system. For the numerical simulations, the Newton polynomial approximation method together with the Atangana--Toufik numerical scheme is employed to obtain approximate solutions for different parameter settings. Furthermore, several visualization techniques, including sensitivity heatmap representation and tornado diagram analysis, are utilized to study the influence of model parameters on the HIV dynamics. The obtained numerical results demonstrate that the proposed fractal--fractional framework provides an effective and reliable approach for analyzing the transient and long-term behavior of HIV transmission dynamics.

 \vspace{.5cm}{\it keywords:  HIV dynamical model; Fractional order; Atangana-Baleanu-Caputo operator; Sensitivity; }


\end{abstract}


\section{Introduction}\label{sec1}


Despite major advances in prevention, diagnosis, and antiretroviral therapy, HIV still one of the world’s most important public health isuse. According to recent global estimates, tens of millions of people are affected with HIV, with the highest burden intense in un- developed countries, particularly in sub-Saharan Africa. Over the past two decades, expanded access to treatment has substantially reduced AIDS-related deaths and improved life expectancy, transforming HIV from a fatal disease into a manageable chronic condition for many patients \cite{s1}. However, new infections continue to occur each year due to limited healthcare access, social stigma, inequality, lack of education, and insufficient preventive measures in vulnerable populations. International organizations and governments continue to emphasize early diagnosis, universal treatment access, public awareness, vaccination research, and the development of long-acting therapies as essential strategies for achieving to vanished such dangerous epidemic disease global is still a major problem in the coming decades \cite{s2}.


HIV is a retrovirus that cannot replicate independently, but it requires the host cells to do so,
reproduce independently. It transports single-stranded RNA, which is transformed.
the virus reverse transcriptase enzyme after into double-stranded DNA.
enters a CD4$^{+}$ T-cell \cite{h1}. The viral DNA becomes incorporated in the host genome,
enabling the infected cell to produce viral RNA and proteins. These components
gather at the cell membrane to create immature virions that bud off the cell.
and grow up by cleavage by protease, and become complete infectious.
particles \cite{h3}. HIV infection can be long term without any clinical manifestations.
period, and symptoms usually develop when the $CD_4^{+}$ T-cell number is decreased.
to almost $200 cells/mm^{3}$ and the viral load grows significantly \cite{h2}. Mathematical modeling has played a big role in comprehending HIV.
pathogenesis and immunodynamics. Wodarz and Nowak \cite{h2} came up with.
model systems of viral development and response to therapy,
and as Mbogo et al. \cite{h1} included stochastic effects and therapeutic.
interventions on in-host HIV dynamics. Arruda et al. \cite{h4} also studied.
In spite of the fact that CD8+ T-cells are known as major agents of the eradication of infected cells and inhibition of viral replication, their numerical effect- especially in late infection- stages remains insufficiently explored. This work is based on these studies and motivated to a mathematical model to explore further that how 
CD8$^{+}$ T-cells controlling HIV dynamics.



Dynamical systems play a fundamental role in bio-mathematics for understanding the evolution, transmission, and control of infectious diseases and biological processes in both integer-order and fractional-order frameworks. Integer-order models are widely used to describe classical biological dynamics, while fractional-order systems provide more realistic descriptions by incorporating memory and hereditary effects, which are essential in many epidemiological and physiological phenomena. Recent studies have demonstrated the effectiveness of these approaches in modeling malaria infection, COVID-19 transmission, and other biomedical processes. In particular, Noeiaghdam and Micula developed efficient strategies to improve the accuracy of nonlinear malaria infection models and fractional COVID-19 systems using stochastic arithmetic techniques and the CADNA library \cite{s4,s6}. Comparative analyses between floating-point arithmetic and discrete stochastic arithmetic for fractional malaria models were investigated in \cite{s5}, while wavelet-based numerical procedures for fractional COVID-19 outbreak models were proposed in \cite{s7}. Furthermore, modified SIR and fractional differential equation models have been successfully applied to study COVID-19 transmission dynamics and stability analysis in different countries and environments \cite{s8,s9,s10}. Advanced fuzzy and subdivision-based approaches have also contributed to infection growth analysis and biomedical data interpretation \cite{s11}.  Zhou \textit{et al.} \cite{s20} proposed a differential equation model incorporating a cure rate for infected CD4$^{+}$ T-cells to investigate treatment effects on HIV dynamics. The global stability of an HIV-1 infection model including an eclipse stage of infected cells, which improved the biological realism of the model  analyzed by Buonomo and Vargas-De-León \cite{s21}. Delayed HIV infection systems with nonlinear incidence rates were studied by Cai \textit{et al.} \cite{s22}, where the authors established important global stability results. Furthermore, Xu \cite{s23} investigated an HIV-1 model with saturation infection and intracellular delay, demonstrating how time-delay effects influence the stability and persistence of the infection. Also more studies can be foud in \cite{s24}. These studies highlight the significance of mathematical modeling in providing theoretical insight into HIV dynamics and developing effective strategies for disease control and treatment.




In this paper, we develop a fractal--fractional HIV dynamics by employing the Atangana-Baleanu-Caputo operator with the Mittag-Leffler kernel, which enables the incorporation of memory effects and fractal characteristics into the biological system. The proposed formulation provides a more realistic framework for describing the complex transmission, and progression behavior of HIV infection. We investigate the basic mathematical properties of the model by through suitable fixed-point approaches and functional analysis techniques. In addition, the Hyers--Ulam stability analysis is carried out to verify the stability and robustness of the obtained solutions under small perturbations. For the numerical investigations, Newton polynomial approximation techniques are applied to compute approximate solutions corresponding to different kernel functions and parameter settings. Furthermore, a comprehensive visualization methodology is implemented to analyze the qualitative and quantitative behavior of the model. The simulations are performed using the fractal-fractional HIV model with the Atangana-Baleanu-Caputo operator involving the Mittag-Leffler kernel and solved numerically through the Atangana-Toufik numerical scheme. The obtained results reveal that variations in key epidemiological parameters significantly influence both the transient and long-term dynamics of the HIV system. To provide deeper insight into the parameter effects, a sensitivity heatmap representation is constructed to visualize the influence of parameters on the system behavior, while tornado diagram analysis is employed to identify the most dominant and sensitive parameters affecting the HIV transmission dynamics.


\section{Model Description}\label{sec2}


The proposed model populations are divided into five compartments showing the main components of HIV dynamics: healthy CD4$^{+}$ T-cells ($T$), infected CD4$^{+}$ T-cells ($I$), free HIV virions ($V$), immune response cells ($Z$), corresponding to CD8$^{+}$ T-cells, and activated immune cells ($Z_a$) as 
\begin{equation} \label{2-1}
\begin{cases}
\dfrac{dT}{dt} = \lambda_T - \mu_T T - \chi T V, \\[6pt]
\dfrac{dI}{dt} = \chi T V - \mu_I I - \alpha I Z_a, \\[6pt]
\dfrac{dV}{dt} = \varepsilon_V \mu_I I - \mu_V V, \\[6pt]
\dfrac{dZ}{dt} = \lambda_Z - \mu_Z Z - \beta Z I, \\[6pt]
\dfrac{dZ_a}{dt} = \beta Z I - \mu_{Z_a} Z_a .
\end{cases}
\end{equation}
The model parameters are as follows: the natural mortality rate of healthy CD4\(^+\) T-cells, denoted \(\mu_T\), the natural death rate of infected CD4\(^+\) T-cells, also denoted \(\mu_T\), the infection transmission rate of CD4\(^+\) T-cells upon contact with free virions, \(\lambda\). The clearance rate of free infectious virions, \(\mu_V\). The constant recruitment rate of CD8\(^+\) T-cells into the system, \(\lambda_Z\), the removal rate of infected CD4\(^+\) T-cells mediated by activated CD8\(^+\) T-cells, \(\alpha\),and the natural mortality rate of CD8\(^+\) T-cells, denoted \(\mu_Z\). The natural decay rate of activated CD8\(^+\) (immune response) cells, also indicated by \(\mu_Z\), and the activation rate of CD8\(^+\) T-cells stimulated by the presence of infected CD4\(^+\) T-cells, \(\beta\). The average production rate of HIV virions by each infected CD4\(^+\) T-cell, \(c_V\).
\section{Preliminaries}\label{sec3}
\begin{definition}\label{def2-1}
To define the fractal-fractional derivative of order $\nu_1$ in the Riemann-Liouville terms of functions $\beta(\tilde{t})$ that are not essential differentiable. Where $(\gamma)$ fractional dimension.
Then for $0 \leq
\nu_1,~ \eta \leq 1$, .
\begin{itemize}
    \item the power law kernel is defined as:
\begin{equation}\label{eq2-1}
    _0^{FFP}D_t^{\nu_1, \gamma}  \Phi (\tilde{t}) = \frac{1}{\Gamma (n-\nu_1)} \frac{d^n}{dt^n} \int_0^{\tilde{t}} (t-\zeta_1)^{n-\nu_1-1} \Phi (\zeta_1) d\zeta_1
\end{equation}
where $n-1 < \nu_1,~ \gamma <n \in N$. And 
\begin{equation}\label{eq2-2}
\frac{D \Phi(\zeta_1)}{D \zeta_1^{\eta}} = \lim_{t \rightarrow \zeta_1} \frac{\psi(\tilde{t}) - \Phi(\zeta_1)}{ t^{\eta} - \zeta_1^{\eta}}   
\end{equation}
\item the exponential decay kernel is obtained as:
\begin{equation}\label{eq2-3}
    _0^{FFE}D_t^{\nu_1, \gamma}  \Phi (\tilde{t}) = \frac{M (\nu_1)}{\Gamma (n-\nu_1)} \frac{d^n}{dt^n} \int_0^{\tilde{t}} exp [- \frac{\nu_1}{1-\nu_1}(t-\zeta_1)] \Phi (\zeta_1) d\zeta_1    
\end{equation}

\item with the application of the Mittag-Leffler kernel (MLK) as fellow:
\begin{equation}\label{eq2-4}
    _0^{FFM}D_t^{\nu_1, \gamma}  \Phi (\tilde{t}) = \frac{AB (\nu_1)}{1-\nu_1 } \frac{d^n}{dt^n} \int_0^{\tilde{t}} \Phi (\zeta_1) E_{\nu_1} [- \frac{\nu_1}{1-\nu_1}(t-\zeta_1)^{\nu_1}]  d\zeta_1    
\end{equation}
\end{itemize}
\end{definition}
\begin{definition}\label{def2-2}
For a continuous function $\Phi(\tilde{t})$ on $(a,b)$, then the fractal-fractional
integral of $\phi(\tilde{t})$.
\begin{itemize}
    \item power law kernel is epressed as:
\begin{equation}\label{eq2-5}
    _0^{FFP}I^{\nu_1, \gamma}  \Phi (\tilde{t}) = \frac{1}{\Gamma ( \nu_1)} \int_0^{\tilde{t}} (t-\zeta_1)^{\nu_1-1} \zeta_1^{1-\eta} \Phi (\zeta_1) d\zeta_1
\end{equation}
\item the exponential decay kernel is represented as:
\begin{equation}\label{eq2-6}
    _0^{FFE}I^{\nu_1, \gamma}  \Phi (\tilde{t}) = \frac{\eta (1-\nu_1) t^{\eta -1} \Phi(\tilde{t})}{M(\nu_1)} +  \frac{\nu_1 \eta}{M (\nu_1)}  \int_0^{\tilde{t}} \zeta_1^{\nu_1-1} \Phi (\zeta_1) d\zeta_1    
\end{equation}
\item similarly MLK is:
\begin{equation}\label{eq2-7}
    _0^{FFM}I^{\nu_1, \gamma}  \Phi (\tilde{t}) = \frac{\eta (a-\nu_1) t^{\eta -1} \beta(\tilde{t})}{AB (\eta_1)} + \frac{ \nu_1 \eta }{AB(\nu_1) \Gamma(\nu_1) } \int_0^{\tilde{t}} (t- \zeta_1)^{\nu_1 -1 } \zeta_1^{\eta-1} \Phi (\zeta_1)   d\zeta_1    
\end{equation}
\end{itemize}
\end{definition}
\section{Fractal--Fractional Model Formulation}\label{sec4}
We extend the classical HIV (\ref{2-1})infection model is transformed using MLK to a fractal--fractional framework. This formulation captures memory effects and heterogeneity in the immune response dynamics.
Let $0<\alpha<1$ denote the fractional order and $0<\gamma\leq 1$ the fractal dimension. The fractal--fractional derivative in the Caputo sense with MKL (Atangana--Baleanu type) of a sufficiently smooth function $f(\tilde{t})$ is defined by
\begin{equation}
{}^{FFM}D_{t}^{\alpha,\gamma} f(\tilde{t})
=
\frac{B(\alpha)}{1-\alpha}
\int_{0}^{t}
E_{\alpha}\!\left(-\frac{\alpha}{1-\alpha}(t-\tau)^{\alpha}\right)
\frac{d}{d\tau}\big(f(\tau)\big)\,
\tau^{1-\gamma}\, d\tau,
\end{equation}
where $E_{\alpha}(\cdot)$ denotes the Mittag--Leffler function and $B(\alpha)$ is a normalization constant satisfying $B(0)=B(1)=1$.
Using the above operator, the fractal--fractional HIV infection model (\ref{eq2-1}) is formulated as
\begin{equation}
\label{eq:FF_HIV_model}
\begin{cases}
{}^{FFM}D_{t}^{\alpha,\gamma} T(\tilde{t})
= \lambda_T - (\mu_T + \chi V(\tilde{t}))T(\tilde{t}), \\[6pt]

{}^{FFM}D_{t}^{\alpha,\gamma} I(\tilde{t})
= \chi T(\tilde{t})V(\tilde{t}) - (\mu_I  + \alpha Z_{a}(\tilde{t}))I(\tilde{t}), \\[6pt]

{}^{FFM}D_{t}^{\alpha,\gamma} V(\tilde{t})
= \varepsilon_V \mu_I I(\tilde{t}) - \mu_V V(\tilde{t}), \\[6pt]

{}^{FFM}D_{t}^{\alpha,\gamma} Z(\tilde{t})
= \lambda_Z - (\mu_Z  + \beta  I(\tilde{t}))Z(\tilde{t}), \\[6pt]

{}^{FFM}D_{t}^{\alpha,\gamma} Z_{a}(\tilde{t})
= \beta Z(\tilde{t}) I(\tilde{t}) - \mu_{Z_a} Z_{a}(\tilde{t}).
\end{cases}
\end{equation}
The model \eqref{eq:FF_HIV_model} is supplemented with the following nonnegative history:
\begin{equation}\label{eq:initial}
(T(0),I(0),V(0),Z(0), Z_a(0) )=(T_0,I_0,V_0,Z_0,Z_{a0} )
\end{equation}
where $T_0$, $I_0$, $V_0$, $Z_0$, and $Z_{a0}$ are given constants.
\section{Proposed Model's Existence and Individuality}
Now, we develop, and prove the basic mathematical properties of our proposed model \eqref{eq:FF_HIV_model}. By changing our proposed model (\ref{eq:FF_HIV_model}) into its integral form, and utilize the fixed-point theorem.

\subsection{Existence of the Solution}
We first reformulate the system \eqref{eq:FF_HIV_model} into an analogous integral form with the application of the fractal-fractional integral operator, let us assume $K(t)= \frac{(\alpha-1)\gamma}t^{\gamma-1}{\mathcal{AB(\alpha)}\Gamma(\alpha)}$, and $L = \frac{\alpha\gamma}{\mathcal{AB(\alpha)}\Gamma(\alpha)}$ \\
\begin{equation}
	\label{eq:integral_system}
	\begin{cases}
		T(\tilde{t}) = T(0) + K(\tilde{t})H_1(\tilde{t}, T) + L \int_0^{\tilde{t}} (\tilde{t}-\zeta_1)^{\alpha-1} \zeta_1^{\gamma-1} H_1(\zeta_1, T) d\zeta_1, \\
		I(\tilde{t}) = I(0) + K(\tilde{t}) H_2(\tilde{t}, I) + L \int_0^{\tilde{t}} (\tilde{t}-\zeta_1)^{\alpha-1} \zeta_1^{\gamma-1} H_2(\zeta_1, I) d\zeta_1, \\
		V(\tilde{t}) = V(0) + K(\tilde{t}) H_3(\tilde{t}, V) + L \int_0^{\tilde{t}} (\tilde{t}-\zeta_1)^{\alpha-1} \zeta_1^{\gamma-1} H_3(\zeta_1, V) d\zeta_1, \\
		Z(\tilde{t}) = Z(0) + K(\tilde{t}) H_4(\tilde{t}, Z) + L \int_0^{\tilde{t}} (\tilde{t}-\zeta_1)^{\alpha-1} \zeta_1^{\gamma-1} H_4(\zeta_1, Z) d\zeta_1, \\
		Z_{a}(\tilde{t}) = Z_a(0) + K(\tilde{t}) H_5(\tilde{t}, Z_a) + L \int_0^{\tilde{t}} (\tilde{t}-\zeta_1)^{\alpha-1} \zeta_1^{\gamma-1} H_5(\zeta_1, Z_a) d\zeta_1.
	\end{cases}
\end{equation}
where the kernels $H_j$ represent the right-hand side of the model equations. We define an iterative sequence $T_n(\tilde{t}), I_n(\tilde{t}), V_n(\tilde{t}), Z_n(\tilde{t}), Z_{an}(\tilde{t})$ to show convergence.
\subsection{Definition of the Kernels}
The kernels $H_i(\tilde{t}, \beta(\tilde{t}))$ for $i \in \{1, 2, 3, 4, 5\}$, show the right hand side of the proposed HIV model \eqref{eq:FF_HIV_model}, and $ X = (t, T, I, V, Z, Z_a)$.
In the subsequent existence, and uniqueness analysis, we assume that the state variables are bounded within a Banach space $C[0, b]$. Therefore, there exist positive constants $A_i$ such that:
\begin{theorem}
Assume that the continuous functions \(T, I, V, Z, Z_a \in L[0,1]\) satisfy 
\(\|T\| \leq \phi_1,\; \|I\| \leq \phi_2,\; \|V\| \leq \phi_3,\; \|Z\| \leq \phi_4,\; \|Z_a\| \leq \phi_5\) 
with each \(\phi_j < 1\). Then each kernel \(H_k\;(m=1,\dots,5)\) is Lipschitz and a contraction.
\end{theorem}
\begin{proof}
For two functions $T(\tilde{\tilde{t}})$ and $\hat{T}(\tilde{t})$, we have:
	\begin{align*}
		\|H_1(\tilde{t}, T) - H_1(\tilde{t}, \hat{T})\| &= \|(\lambda_T - \mu_T T - \chi TV) - (\lambda_T - \mu_T \hat{T} - \chi \hat{T}V)\| \\
		&= \|-\mu_T(T - \hat{T}) - \chi V(T - \hat{T})\| \\
		&\leq [\mu_T + \chi \|V\|] \|T - \hat{T}\| \\
		&\leq [\mu_T + \chi A_3] \|T - \hat{T}\| \\
		&\leq \phi_1 \|T - \hat{T}\|,
	\end{align*}
	where $\phi_1 = \mu_T + \chi A_3 < 1$. Thus, $H_1$ fulfills the Lipschitz condition.
	
	Next, for $H_2(\tilde{t}, I)$ with $I(\tilde{t})$ and $\hat{I}(\tilde{t})$:
	\begin{align*}
		\|H_2(\tilde{t}, I) - H_2(\tilde{t}, \hat{I})\| &= \|(\chi TV - \mu_I I - \alpha I Z_a) - (\chi TV - \mu_I \hat{I} - \alpha \hat{I} Z_a)\| \\
		&\leq [\mu_I + \alpha \|Z_a\|] \|I - \hat{I}\| \\
		&\leq [\mu_I + \alpha A_5] \|I - \hat{I}\| \\
		&\leq \phi_2 \|I - \hat{I}\|,
	\end{align*}
	where $\phi_2 = \mu_I + \alpha A_5 < 1$.
	
	For $H_3(\tilde{t}, V)$ with $V(\tilde{t})$ and $\hat{V}(\tilde{t})$:
	\begin{align*}
		\|H_3(\tilde{t}, V) - H_3(\tilde{t}, \hat{V})\| &= \|(\varepsilon_V \mu_I I - \mu_V V) - (\varepsilon_V \mu_I I - \mu_V \hat{V})\| \\
		&\leq \mu_V \|V - \hat{V}\| \\
		&\leq \phi_3 \|V - \hat{V}\|,
	\end{align*}
	where $\phi_3 = \mu_V < 1$.
	
	For $H_4(\tilde{t}, Z)$ with $Z(\tilde{t})$ and $\hat{Z}(\tilde{t})$:
	\begin{align*}
		\|H_4(\tilde{t}, Z) - H_4(\tilde{t}, \hat{Z})\| &= \|(\lambda_Z - \mu_Z Z - \beta Z I) - (\lambda_Z - \mu_Z \hat{Z} - \beta \hat{Z} I)\| \\
		&\leq [\mu_Z + \beta \|I\|] \|Z - \hat{Z}\| \\
		&\leq [\mu_Z + \beta A_2] \|Z - \hat{Z}\| \\
		&\leq \phi_4 \|Z - \hat{Z}\|,
	\end{align*}
	where $\phi_4 = \mu_Z + \beta A_2 < 1$.
	
	Finally, for $H_5(\tilde{t}, Z_a)$ with $Z_{a}(\tilde{t})$ and $\hat{Z}_a(\tilde{t})$:
	\begin{align*}
		\|H_5(\tilde{t}, Z_a) - H_5(\tilde{t}, \hat{Z}_a)\| &= \|(\beta Z I - \mu_{Z_a} Z_a) - (\beta Z I - \mu_{Z_a} \hat{Z}_a)\| \\
		&\leq \mu_{Z_a} \|Z_a - \hat{Z}_a\| \\
		&\leq \phi_5 \|Z_a - \hat{Z}_a\|,
	\end{align*}
	where $\phi_5 = \mu_{Z_a} < 1$.
    When $\phi_j < 1$ then all the kernels verified the desire conditions.  the proof is completed.
\end{proof}

\subsection{Iterative Scheme for Qualitative Analysis}
Using of the \eqref{eq:integral_system} to determine the existence of solutions, we define the following iterative convergent sequence:
\begin{equation}
	\begin{cases}
		T_n(\tilde{t}) = K(\tilde{t}) H_1(\tilde{t}, T_{n-1}) + L \int_0^{\tilde{t}} (\tilde{t}-\zeta_1)^{\alpha-1} \zeta_1^{\gamma-1} H_1(\zeta_1, T_{n-1}) d\zeta_1, \\[10pt]
		I_n(\tilde{t}) = K(\tilde{t}) H_2(\tilde{t}, I_{n-1}) + L \int_0^{\tilde{t}} (\tilde{t}-\zeta_1)^{\alpha-1} \zeta_1^{\gamma-1} H_2(\zeta_1, I_{n-1}) d\zeta_1, \\[10pt]
		V_n(\tilde{t}) = K(\tilde{t}) H_3(\tilde{t}, V_{n-1}) + L \int_0^{\tilde{t}} (\tilde{t}-\zeta_1)^{\alpha-1} \zeta_1^{\gamma-1} H_3(\zeta_1, V_{n-1}) d\zeta_1, \\[10pt]
		Z_n(\tilde{t}) = K(\tilde{t}) H_4(\tilde{t}, Z_{n-1}) + L \int_0^{\tilde{t}} (\tilde{t}-\zeta_1)^{\alpha-1} \zeta_1^{\gamma-1} H_4(\zeta_1, Z_{n-1}) d\zeta_1, \\[10pt]
		Z_{an}(\tilde{t}) = K(\tilde{t}) H_5(\tilde{t}, Z_{an-1}) + L \int_0^{\tilde{t}} (\tilde{t}-\zeta_1)^{\alpha-1} \zeta_1^{\gamma-1} H_5(\zeta_1, Z_{an-1}) d\zeta_1.
	\end{cases}
\end{equation}
Next, we consider the differences between successive terms of the iterative sequence as follows:
\begin{align*}
	\Delta T_{n}(\tilde{t}) &= T_{n}(\tilde{t}) - T_{n-1}(\tilde{t}) \\
	&= K(\tilde{t}) \left( H_1(\tilde{t}, T_{n-1}) - H_1(\tilde{t}, T_{n-2}) \right) \\
	&\quad + L \int_0^{\tilde{t}} (\tilde{t}-\zeta_1)^{\alpha-1} \zeta_1^{\gamma-1} \left( H_1(\zeta_1, T_{n-1}) - H_1(\zeta_1, T_{n-2}) \right) d\zeta_1, \\[10pt]
	\Delta I_{n}(\tilde{t}) &= I_{n}(\tilde{t}) - I_{n-1}(\tilde{t}) \\
	&= K(\tilde{t}) \left( H_2(t, I_{n-1}) - H_2(\tilde{t}, I_{n-2}) \right) \\
	&\quad + L \int_0^{\tilde{t}} (\tilde{t}-\zeta_1)^{\alpha-1} \zeta_1^{\gamma-1} \left( H_2(\zeta_1, I_{n-1}) - H_2(\zeta_1, I_{n-2}) \right) d\zeta_1, \\[10pt]
	\Delta V_{n}(\tilde{t}) &= V_{n}(\tilde{t}) - V_{n-1}(\tilde{t}) \\
	&= K(\tilde{t}) \left( H_3(\tilde{t}, V_{n-1}) - H_3(\tilde{t}, V_{n-2}) \right) \\
	&\quad + L \int_0^{\tilde{t}} (\tilde{t}-\zeta_1)^{\alpha-1} \zeta_1^{\gamma-1} \left( H_3(\zeta_1, V_{n-1}) - H_3(\zeta_1, V_{n-2}) \right) d\zeta_1, \\[10pt]
\end{align*}
   \begin{align*} 
	\Delta Z_{n}(\tilde{t}) &= Z_{n}(\tilde{t}) - Z_{n-1}(\tilde{t}) \\
	&= K(\tilde{t}) \left( H_4(\tilde{t}, Z_{n-1}) - H_4(\tilde{t}, Z_{n-2}) \right) \\
	&\quad + L \int_0^{\tilde{t}} (\tilde{t}-\zeta_1)^{\alpha-1} \zeta_1^{\gamma-1} \left( H_4(\zeta_1, Z_{n-1}) - H_4(\zeta_1, Z_{n-2}) \right) d\zeta_1, \\[10pt]
	\Delta Z_{an}(\tilde{t}) &= Z_{an}(\tilde{t}) - Z_{an-1}(\tilde{t}) \\
	&= K(\tilde{t}) \left( H_5(\tilde{t}, Z_{an-1}) - H_5(\tilde{t}, Z_{an-2}) \right) \\
	&\quad + L \int_0^{\tilde{t}} (\tilde{t}-\zeta_1)^{\alpha-1} \zeta_1^{\gamma-1} \left( H_5(\zeta_1, Z_{an-1}) - H_5(\zeta_1, Z_{an-2}) \right) d\zeta_1.
\end{align*}
Taking the norms of the aforementioned system on both sides, we obtained:
\begin{align*}
	\|\Delta T_{n+1}(\tilde{t})\| &= K(\tilde{t}) \|H_1(\tilde{t}, T_n) - H_1(\tilde{t}, T_{n-1})\| \\
	&\quad + L \int_0^{\tilde{t}} (\tilde{t}-\zeta_1)^{\alpha-1} \zeta_1^{\gamma-1} \|H_1(\zeta_1, T_n) - H_1(\zeta_1, T_{n-1})\| d\zeta_1, \\[10pt]
	\|\Delta I_{n+1}(\tilde{t})\| &= K(\tilde{t}) \|H_2(\tilde{t}, I_n) - H_2(\tilde{t}, I_{n-1})\| \\
	&\quad + L \int_0^{\tilde{t}} (\tilde{t}-\zeta_1)^{\alpha-1} \zeta_1^{\gamma-1} \|H_2(\zeta_1, I_n) - H_2(\zeta_1, I_{n-1})\| d\zeta_1, \\[10pt]
	\|\Delta V_{n+1}(\tilde{t})\| &= K(\tilde{t}) \|H_3(\tilde{t}, V_n) - H_3(\tilde{t}, V_{n-1})\| \\
	&\quad + L \int_0^{\tilde{t}} (\tilde{t}-\zeta_1)^{\alpha-1} \zeta_1^{\gamma-1} \|H_3(\zeta_1, V_n) - H_3(\zeta_1, V_{n-1})\| d\zeta_1, \\[10pt]
	\|\Delta Z_{n+1}(\tilde{t})\| &= K(\tilde{t}) \|H_4(\tilde{t}, Z_n) - H_4(\tilde{t}, Z_{n-1})\| \\
	&\quad + L \int_0^{\tilde{t}} (\tilde{t}-\zeta_1)^{\alpha-1} \zeta_1^{\gamma-1} \|H_4(\zeta_1, Z_n) - H_4(\zeta_1, Z_{n-1})\| d\zeta_1, \\[10pt]
	\|\Delta Z_{an+1}(\tilde{t})\| &= K(\tilde{t}) \|H_5(\tilde{t}, Z_{an}) - H_5(\tilde{t}, Z_{an-1})\| \\
	&\quad + L \int_0^{\tilde{t}} (\tilde{t}-\zeta_1)^{\alpha-1} \zeta_1^{\gamma-1} \|H_5(\zeta_1, Z_{an}) - H_5(\zeta_1, Z_{an-1})\| d\zeta_1.
\end{align*}

By applying the Lipschitz conditions $\|H_j(\Psi_n) - H_j(\Psi_{n-1})\| \leq \phi_j \|\Delta \Psi_n\|$, we obtain the generalized bounding system by taking $C(t) =\left[ K(\tilde{t}) + \frac{\alpha\gamma \Gamma(\gamma) t^{\alpha+\gamma-1}}{\mathcal{AB(\alpha)}\Gamma(\alpha+\gamma)} \right] $.
\begin{align*}
	\|\Delta T_{n+1}(\tilde{t})\| &\leq C(\tilde{t}) \phi_1 \|\Delta T_n(\tilde{t})\|, \\[10pt]
	\|\Delta I_{n+1}(\tilde{t})\| &\leq C(\tilde{t}) \phi_2 \|\Delta I_n(\tilde{t})\|, \\[10pt]
	\|\Delta V_{n+1}(\tilde{t})\| &\leq C(\tilde{t}) \phi_3 \|\Delta V_n(\tilde{t})\|, \\[10pt]
	\|\Delta Z_{n+1}(\tilde{t})\| &\leq C(\tilde{t}) \phi_4 \|\Delta Z_n(\tilde{t})\|, \\[10pt]
	\|\Delta Z_{an+1}(\tilde{t})\| &\leq C(\tilde{t}) \phi_5 \|\Delta Z_{an}(\tilde{t})\|.
\end{align*}
\begin{theorem}
	The solution of our proposed model \eqref{eq:FF_HIV_model} is exist if:
	\begin{equation}
		\Omega = \max[\phi_1, \phi_2, \phi_3, \phi_4, \phi_5] < 1.
	\end{equation}
\end{theorem}

\begin{proof}
	We define the difference functions between the iterative sequence and the actual solution as follows:
	\begin{equation*}
		\begin{cases}
			\mathcal{Z}_{1n}(\tilde{t}) = T_{n+1}(\tilde{t}) - T(\tilde{t}), \\
			\mathcal{Z}_{2n}(\tilde{t}) = I_{n+1}(\tilde{t}) - I(\tilde{t}), \\
			\mathcal{Z}_{3n}(\tilde{t}) = V_{n+1}(\tilde{t}) - V(\tilde{t}), \\
			\mathcal{Z}_{4n}(\tilde{t}) = Z_{n+1}(\tilde{t}) - Z(\tilde{t}), \\
			\mathcal{Z}_{5n}(\tilde{t}) = Z_{a,n+1}(\tilde{t}) - Z_{a}(\tilde{t}).
		\end{cases}
	\end{equation*}
	
	Applying the norm of the first equation $\mathcal{Z}_{1n}(\tilde{t})$ and applying the Lipschitz criterion $\phi_1$, we get:
	\begin{align*}
		\|\mathcal{Z}_{1n}(\tilde{t})\| &= \|T_{n+1}(\tilde{t}) - T(\tilde{t})\| \\
		&\leq C(\tilde{t}) \phi_1 \|T_n - T\| \\
		&\leq C(\tilde{t})^n \lambda^n \|T_n - T\|,
	\end{align*}
	where $\lambda < 1$. As $n \to \infty$, we observe that $\|\mathcal{Z}_{1n}(\tilde{t})\| \to 0$. Similarly, for the remaining compartments:
	\begin{align*}
		\|\mathcal{Z}_{2n}\| &\leq C(\tilde{t})^n \lambda^n \|I_n - I\| \to 0, \\
		\|\mathcal{Z}_{3n}\| &\leq C(\tilde{t})^n \lambda^n \|V_n - V\| \to 0, \\
		\|\mathcal{Z}_{4n}\| &\leq C(\tilde{t})^n \lambda^n \|Z_n - Z\| \to 0, \\
		\|\mathcal{Z}_{5n}\| &\leq C(\tilde{t})^n \lambda^n \|Z_{an} - Z_a\| \to 0.
	\end{align*}
	
	Which show that $\mathcal{Z}_{jn}(\tilde{t}) \to 0$ as $n \to \infty$ for $j \in \{1, 2, 3, 4, 5\}$, which confirms that a solution exists for the HIV model.
\end{proof}
\begin{theorem}
	If $C(\tilde{t}) \phi_j < 1, \quad j \in \{1, 2, 3, 4, 5\}$, then Then our proposed model \eqref{eq:FF_HIV_model} has a unique solution.
\end{theorem}

\begin{proof}
 Let $\{T_1, I_1, V_1, Z_1, Z_a\}$ and $\{\tilde{T_1}, \tilde{I_1}, \tilde{V_1}, \tilde{Z_1}, \tilde{Z}_a\}$ be two distinct solutions of the system. 
	
	Taking the difference of $T_1(\tilde{t})$, and $\tilde{T_1}(\tilde{t})$ then applying the norm, we obtain:
	\begin{align*}
		\|T_1(\tilde{t}) - \tilde{T_1}(\tilde{t})\| &= \Bigg\| K(\tilde{t}) \left( H_1(\tilde{t}, T) - H_1(\tilde{t}, \tilde{T}) \right) \\
		&\quad + L \int_0^{\tilde{t}} (\tilde{t}-\zeta_1)^{\alpha-1} \zeta_1^{\gamma-1} \left( H_1(\zeta_1, T) - H_1(\zeta_1, \tilde{T}) \right) d\zeta_1 \Bigg\|.
	\end{align*}
	
	Applying the triangle inequality and the Lipschitz criterion $\phi_1$:
	\begin{align*}
		\|T_1(\tilde{t}) - \tilde{T_1}(\tilde{t})\| &\leq K(\tilde{t}) \phi_1 \|T - \tilde{T}\| \\
		&\quad + L \phi_1 \|T_1 - \tilde{T_1}\| \int_0^{\tilde{t}} (\tilde{t}-\zeta_1)^{\alpha-1} \zeta_1^{\gamma-1} d\zeta_1.
	\end{align*}
	
	Evaluating the integral $\int_0^{\tilde{t}} (\tilde{t}-\zeta_1)^{\alpha-1} \zeta_1^{\gamma-1} d\zeta_1 = \frac{\Gamma(\alpha)\Gamma(\gamma)}{\Gamma(\alpha+\gamma)} t^{\alpha+\gamma-1}$, we obtain:
	\begin{equation*}
		\|T_1(\tilde{t}) - \tilde{T_1}(\tilde{t})\| \leq C(\tilde{t}) \phi_1 \|T_1 - \tilde{T_1}\|.
	\end{equation*}
	
	Rearranging the terms leads to:
	\begin{equation*}
		\|T_1 - \tilde{T_1}\| - C(\tilde{t}) \phi_1 \|T_1 - \tilde{T_1}\| \leq 0,
	\end{equation*}
	\begin{equation}
		( 1 - C(\tilde{t}) )\|T_1 - \tilde{T_1}\| \leq 0.
	\end{equation}
	
	The inequality mentioned above is accurate if and only if:
	\begin{equation*}
		\|T_1 - \tilde{T_1}\| = 0 \implies T_1 = \tilde{T_1}.
	\end{equation*}
	
	Similarly, by repeating the same steps for the other compartments, we obtain $I_1 = \tilde{I_1},V_1 = \tilde{V_1},Z_1 = \tilde{Z_1},$ and $Z_a = \tilde{Z}_a.$
	Therefore, the system \eqref{eq:FF_HIV_model} has a unique solution.
\end{proof}
\subsection{Hyers-Ulam Stability Analysis}
In this subsection, we investigate the Hyers-Ulam (HU) stability of the proposed fractal-fractional HIV model \eqref{eq:FF_HIV_model}. We begin by defining the stability criteria for the system.
\subsection{Basic Definitions}
Let $\epsilon > 0$ be a small positive constant, and let the following inequality hold for the fractal-fractional operator:
\begin{equation}
	\left\| {}^{FFM}D_{\tilde{t}}^{\alpha,\gamma} \beta(\tilde{t}) - H_j(\tilde{t}, \beta(\tilde{t})) \right\| \leq \epsilon, \quad j \in \{1, \dots, 5\}.
\end{equation}
A solution $\beta(\tilde{t})$ of the model is said to be Hyers-Ulam stable if there exists a constant $C > 0$ and an exact solution $\beta^*(\tilde{t})$ such that:
\begin{equation}
	\|\beta(\tilde{t}) - \beta^*(\tilde{t})\| \leq C \epsilon.
\end{equation}
\begin{theorem}\label{thm:6.1}
Let $H_j$ be the kernels of the system \eqref{eq:FF_HIV_model}. If $\beta(\tilde{t})$ is an approximate solution satisfying the inequality:
\begin{equation}
    \left\| {}^{FFM}D_{\tilde{t}}^{\alpha,\gamma} \beta
    (\tilde{t}) - H_j(\tilde{t}, \beta(\tilde{t})) \right\| \leq \epsilon, \quad \epsilon > 0,
\end{equation}
then $\beta(\tilde{t})$ satisfies the following integral inequality:
\begin{equation}
    \left\| \beta(\tilde{t}) - \left( \beta(0) + \mathcal{I}_{FF}^{\alpha, \gamma} H_j(\tilde{t}, \beta(\tilde{t})) \right) \right\| \leq C(\tilde{t}) \epsilon = \Lambda \epsilon.
\end{equation}
\end{theorem}

\begin{proof}
By integrating the perturbed differential inequality, we obtain:
\begin{align*}
    \beta(\tilde{t}) &\leq \beta(0) + K(\tilde{t}) (H_j(t, \Psi(\tilde{t})) + \epsilon) \\
    &\quad + L \int_0^{\tilde{t}} (\tilde{t}-\zeta_1)^{\alpha-1} \zeta_1^{\gamma-1} (H_j(\zeta_1, \beta(\zeta_1)) + \epsilon) d\zeta_1.
\end{align*}
Applying the norm and subtracting the exact integral form $\beta(0) + \mathcal{I}_{FF}^{\alpha, \gamma} H_j(\tilde{t}, \beta(\tilde{t}))$ from both sides:
\begin{align*}
    \left\| \Psi(\tilde{t}) - \left( \beta(0) + \mathcal{I}_{FF}^{\alpha, \gamma} H_j(\tilde{t}, \Psi(\tilde{t})) \right) \right\| &\leq \epsilon \left[ K(\tilde{t}) + L \int_0^{\tilde{t}} (\tilde{t}-\zeta_1)^{\alpha-1} \zeta_1^{\gamma-1} d\zeta_1 \right] \\
    &\leq \epsilon C(\tilde{t}).
\end{align*}
By defining the term as $\Lambda = C(\tilde{t})$, we arrive at:
\begin{equation*}
    \left\| \beta(\tilde{t}) - \left( \beta(0) + \mathcal{I}_{FF}^{\alpha, \gamma} H_j(\tilde{t}, \beta(\tilde{t})) \right) \right\| \leq \Lambda \epsilon.
\end{equation*}
\end{proof}
\begin{theorem}
	The fractal-fractional HIV model \eqref{eq:FF_HIV_model} is Hyers-Ulam stable if the condition $\phi_j \Lambda < 1$ holds for all $j \in \{1, 2, 3, 4, 5\}$, where $\phi_j$ are the Lipschitz constants and $\Lambda$ is the integral bound defined in Theorem \ref{thm:6.1}.
\end{theorem}

\begin{proof}
	Let $\beta(\tilde{t}) = \{T-1, I_1, V_1, Z_1, Z_a\}$ be the approximate solution of the system satisfying the perturbed inequality, and let $\beta^*(\tilde{t}) = \{T_1^*, I_1^*, V_1^*, Z_1^*, Z_a^*\}$ be the unique exact solution of the system \eqref{eq:FF_HIV_model}.
	
	Starting with the first compartment, $T(\tilde{t})$, we consider the difference:
	\begin{align*}
		\|T_1(\tilde{t}) - T_1^*(\tilde{t})\| &= \left\| T_1(\tilde{t}) - \left( T_(0) + \mathcal{I}_{FF}^{\alpha, \gamma} H_1(\tilde{t}, T_1^*) \right) \right\| \\
		&\leq \left\| T_1(\tilde{t}) - \left( T_1(0) + \mathcal{I}_{FF}^{\alpha, \gamma} H_1(\tilde{t}, T_1) \right) \right\| + \left\| \mathcal{I}_{FF}^{\alpha, \gamma} H_1(\tilde{t}, T_1) - \mathcal{I}_{FF}^{\alpha, \gamma} H_1(\tilde{t}, T_1^*) \right\|.
	\end{align*}
	
	Using the result from Theorem \ref{thm:6.1} for the first term and the Lipschitz condition for the second term:
	\begin{equation*}
		\|T_1(\tilde{t}) - T_1^*(\tilde{t})\| \leq \Lambda \epsilon + C(\tilde{t}) \phi_1 \|T_1(\tilde{t}) - T_1^*(\tilde{t})\|.
	\end{equation*}
	
	Applying the notation $\Lambda = K(\tilde{t}) + \frac{\alpha\gamma \Gamma(\gamma) t^{\alpha+\gamma-1}}{\mathcal{AB(\alpha)}\Gamma(\alpha+\gamma)}$, the inequality becomes:
	\begin{equation*}
		\|T_1(\tilde{t}) - T_1^*(\tilde{t})\| \leq \Lambda \epsilon + \phi_1 \Lambda \|T_1(\tilde{t}) - T_1^*(\tilde{t})\|.
	\end{equation*}
	
	Rearranging the terms to isolate the norm of the difference:
	\begin{equation*}
		\|T_1(\tilde{t}) - T_1^*(\tilde{t})\| (1 - \phi_1 \Lambda) \leq \Lambda \epsilon,
	\end{equation*}
	which leads to:
	\begin{equation}
		\|T_1(\tilde{t}) - T_1^*(\tilde{t})\| \leq \frac{\Lambda}{1 - \phi_1 \Lambda} \epsilon = C_1 \epsilon.
	\end{equation}
	
	By repeating the same logical steps for the remaining compartments $I_1, V_1, Z_1,$ and $Z_a$, we obtain:
	\begin{align*}
		\|I_1(\tilde{t}) - I_1^*(\tilde{t})\| &\leq C_2 \epsilon, \\
		\|V_1(\tilde{t}) - V_1^*(\tilde{t})\| &\leq C_3 \epsilon, \\
		\|Z_1(\tilde{t}) - Z_1^*(\tilde{t})\| &\leq C_4 \epsilon, \\
		\|Z_{a}(\tilde{t}) - Z_a^*(\tilde{t})\| &\leq C_5 \epsilon,
	\end{align*}
	where $C_j = \frac{\Lambda}{1 - \phi_j \Lambda} > 0$. Since all differences are bounded by a constant multiple of $\epsilon$, the proposed fractal-fractional HIV model is Hyers-Ulam stable.
\end{proof}
\section{Numerical Methodology}\label{sec7}
In order to numerically solve the proposed fractal-fractional HIV model \eqref{eq:FF_HIV_model} under the Mittag-Leffler kernel framework, we describe a numerical technique based on a Newton polynomial method \cite{7}. This approach captures non-local historical memory profiles and fractal characteristics simultaneously. For convenience and clarity of the formulation, we rewrite the system \eqref{eq:FF_HIV_model} as follows:
\begin{equation}
\label{eq:functional_form}
\begin{cases}
{}^{FFM}_{0}D_{t}^{\alpha,\gamma} T(\tilde{t}) = \mathcal{F}_1X, \\[6pt]
{}^{FFM}_{0}D_{t}^{\alpha,\gamma} I(\tilde{t}) = \mathcal{F}_2X, \\[6pt]
{}^{FFM}_{0}D_{t}^{\alpha,\gamma} V(\tilde{t}) = \mathcal{F}_3X, \\[6pt]
{}^{FFM}_{0}D_{t}^{\alpha,\gamma} Z(\tilde{t}) = \mathcal{F}_4X, \\[6pt]
{}^{FFM}_{0}D_{t}^{\alpha,\gamma} Z_{a}(\tilde{t}) = \mathcal{F}_5X,
\end{cases}
\end{equation}
\begin{equation*}
X = (\tilde{t}, T, I, V, Z, Z_a)
\end{equation*}
where the nonlinear vector fields representing the transmission dynamics are defined as:
\begin{equation}
\begin{cases}
\mathcal{F}_1 = \lambda_T - (\mu_T + \chi V(\tilde{t}))T(\tilde{t}), \\[6pt]
\mathcal{F}_2 = \chi T(\tilde{t})V(\tilde{t}) - (\mu_I  + \alpha Z_{a}(\tilde{t}))I(\tilde{t}), \\[6pt]
\mathcal{F}_3 = \varepsilon_V \mu_I I(\tilde{t}) - \mu_V V(\tilde{t}), \\[6pt]
\mathcal{F}_4 = \lambda_Z - (\mu_Z  + \beta  I(\tilde{t}))Z(\tilde{t}), \\[6pt]
\mathcal{F}_5 = \beta Z(\tilde{t}) I(\tilde{t}) - \mu_{Z_a} Z_{a}(\tilde{t}).
\end{cases}
\end{equation}
The following results are obtained by using a fractal-fractional integral operator with a non-singular Mittag-Leffler law kernel, transforming the model into its corresponding Volterra integral form at the mesh point $t_{n+1}$:
\begin{equation}
\begin{aligned}
T(t_{n+1}) &= \frac{1-\alpha}{\mathbb{AB}(\alpha)} t_n^{\gamma-1} \mathcal{F}_1(t_n, P^n) + M \sum_{m=2}^{n} \int_{t_m}^{t_{m+1}} \mathcal{F}_1(\tau, P) \tau^{\gamma-1} (t_{n+1} - \tau)^{\alpha-1} d\tau, \\[6pt]
I(t_{n+1}) &= \frac{1-\alpha}{\mathbb{AB}(\alpha)} t_n^{\gamma-1} \mathcal{F}_2(t_n, P^n) + M \sum_{m=2}^{n} \int_{t_m}^{t_{m+1}} \mathcal{F}_2(\tau, P) \tau^{\gamma-1} (t_{n+1} - \tau)^{\alpha-1} d\tau, \\[6pt]
V(t_{n+1}) &= \frac{1-\alpha}{\mathbb{AB}(\alpha)} t_n^{\gamma-1} \mathcal{F}_3(t_n, P^n) + M \sum_{m=2}^{n} \int_{t_m}^{t_{m+1}} \mathcal{F}_3(\tau, P) \tau^{\gamma-1} (t_{n+1} - \tau)^{\alpha-1} d\tau, \\[6pt]
Z(t_{n+1}) &= \frac{1-\alpha}{\mathbb{AB}(\alpha)} t_n^{\gamma-1} \mathcal{F}_4(t_n, P^n) + M \sum_{m=2}^{n} \int_{t_m}^{t_{m+1}} \mathcal{F}_4(\tau, P) \tau^{\gamma-1} (t_{n+1} - \tau)^{\alpha-1} d\tau, \\[6pt]
Z_a(t_{n+1}) &= \frac{1-\alpha}{\mathbb{AB}(\alpha)} t_n^{\gamma-1} \mathcal{F}_5(t_n, P^n) + M \sum_{m=2}^{n} \int_{t_m}^{t_{m+1}} \mathcal{F}_5(\tau, P) \tau^{\gamma-1} (t_{n+1} - \tau)^{\alpha-1} d\tau,
\end{aligned}
\end{equation}
where
\begin{equation*}
    M = \frac{\alpha}{\mathbb{AB}(\alpha)\Gamma(\alpha)}.
\end{equation*}
$\mathbb{AB}(\alpha)$ is the normalization function of the Atangana-Baleanu kernel, and $P = \{T, I, V, Z, Z_a\}$ denotes the state vector layout. With the help of Newton polynomial, we approximate the combined continuous integrand function $\mathcal{G}(\tau, P) = \tau^{\gamma-1}\mathcal{F}_i(\tau, P)$ over the subinterval $[t_m, t_{m+1}]$ utilizing previous history keys $\{t_{m-2}, t_{m-1}, t_m\}$:
\begin{equation}
\begin{aligned}
\mathcal{G}(\tau, P) \simeq & \,\, \mathcal{G}(t_{m-2}, P^{m-2}) + \frac{1}{\Delta t} \left[ \mathcal{G}(t_{m-1}, P^{m-1}) - \mathcal{G}(t_{m-2}, P^{m-2}) \right] (\tau - t_{m-2}) \\
& + \frac{1}{2(\Delta t)^2} \left[ \mathcal{G}(t_m, P^m) - 2\mathcal{G}(t_{m-1}, P^{m-1}) + \mathcal{G}(t_{m-2}, P^{m-2}) \right] (\tau - t_{m-2})(\tau - t_{m-1}).
\end{aligned}
\end{equation}
For the integrals in the above formulation, we perform the following exact analytical integrations:
\begin{equation}
\begin{aligned}
\int_{t_m}^{t_{m+1}} (t_{n+1} - \tau)^{\alpha-1} d\tau &= \frac{(\Delta t)^\alpha}{\alpha} \left[ \Pi_k^n \right], \\[6pt]
\int_{t_m}^{t_{m+1}} (\tau - t_{m-2})(t_{n+1} - \tau)^{\alpha-1} d\tau &= \frac{(\Delta t)^{\alpha+1}}{\alpha(\alpha+1)} \left[ \Sigma_k^n \right], \\[6pt]
\int_{t_m}^{t_{m+1}} (\tau - t_{m-2})(\tau - t_{m-1})(t_{n+1} - \tau)^{\alpha-1} d\tau &= \frac{(\Delta t)^{\alpha+2}}{\alpha(\alpha+1)(\alpha+2)} \left[ \Omega_k^n \right],
\end{aligned}
\end{equation}
where the specialized algebraic weight constants tracking the historical trajectories are evaluated as:
\begin{equation}
\begin{aligned}
\Pi_k^n = & \, (n - m + 1)^\alpha - (n - m)^\alpha, \\
\Sigma_k^n = & \, (n - m + 1)^\alpha (n - m + 3 + 2\alpha) - (n - m)^\alpha (n - m + 3 + 3\alpha), \\
\Omega_k^n = & \, (n - m + 1)^\alpha \left\{ 2(n - m)^2 + (3\alpha + 10)(n - m) + 2\alpha^2 + 9\alpha + 12 \right\} \\
& - (n - m)^\alpha \left\{ 2(n - m)^2 + (5\alpha + 10)(n - m) + 6\alpha^2 + 18\alpha + 12 \right\}.
\end{aligned}
\end{equation}
Replacing the Newton polynomial terms back into the integrated architecture yields the complete multi-step numerical scheme for the healthy target CD4+ T-cell population $T(t_{n+1})$:
\begin{equation}
\begin{aligned}
T(t_{n+1}) = & \, \frac{1-\alpha}{\mathbb{AB}(\alpha)} t_n^{\gamma-1} \mathcal{F}_1(t_n, P^n) + \frac{(\Delta t)^\alpha}{\mathbb{AB}(\alpha)\Gamma(\alpha+1)} \sum_{m=2}^{n} t_{m-2}^{\gamma-1} \mathcal{F}_1(t_{m-2}, P^{m-2}) \times \left[ \Pi_k^n \right] \\
& + \frac{(\Delta t)^\alpha}{\mathbb{AB}(\alpha)\Gamma(\alpha+2)} \sum_{m=2}^{n} \left[ t_{m-1}^{\gamma-1} \mathcal{F}_1(t_{m-1}, P^{m-1}) - t_{m-2}^{\gamma-1} \mathcal{F}_1(t_{m-2}, P^{m-2}) \right] \times \left[ \Sigma_k^n \right] \\
& + \frac{(\Delta t)^\alpha}{2\mathbb{AB}(\alpha)\Gamma(\alpha+3)} \sum_{m=2}^{n} \left[ t_m^{\gamma-1} \mathcal{F}_1(t_m, P^m) - 2t_{m-1}^{\gamma-1} \mathcal{F}_1(t_{m-1}, P^{m-1}) + t_{m-2}^{\gamma-1} \mathcal{F}_1(t_{m-2}, P^{m-2}) \right] \times \left[ \Omega_k^n \right].
\end{aligned}
\end{equation}
By using the same process on the remaining structural compartments of the HIV dynamics, we complete the numerical scheme for computational simulations:
\begin{equation}
\begin{aligned}
I(t_{n+1}) = & \, \frac{1-\alpha}{\mathbb{AB}(\alpha)} t_n^{\gamma-1} \mathcal{F}_2(t_n, P^n) + \frac{(\Delta t)^\alpha}{\mathbb{AB}(\alpha)\Gamma(\alpha+1)} \sum_{m=2}^{n} t_{m-2}^{\gamma-1} \mathcal{F}_2(t_{m-2}, P^{m-2}) \times \left[ \Pi_k^n \right] \\
& + \frac{(\Delta t)^\alpha}{\mathbb{AB}(\alpha)\Gamma(\alpha+2)} \sum_{m=2}^{n} \left[ t_{m-1}^{\gamma-1} \mathcal{F}_2(t_{m-1}, P^{m-1}) - t_{m-2}^{\gamma-1} \mathcal{F}_2(t_{m-2}, P^{m-2}) \right] \times \left[ \Sigma_k^n \right] \\
& + \frac{(\Delta t)^\alpha}{2\mathbb{AB}(\alpha)\Gamma(\alpha+3)} \sum_{m=2}^{n} \left[ t_m^{\gamma-1} \mathcal{F}_2(t_m, P^m) - 2t_{m-1}^{\gamma-1} \mathcal{F}_2(t_{m-1}, P^{m-1}) + t_{m-2}^{\gamma-1} \mathcal{F}_2(t_{m-2}, P^{m-2}) \right] \times \left[ \Omega_k^n \right],
\end{aligned}
\end{equation}

\begin{equation}
\begin{aligned}
V(t_{n+1}) = & \, \frac{1-\alpha}{\mathbb{AB}(\alpha)} t_n^{\gamma-1} \mathcal{F}_3(t_n, P^n) + \frac{(\Delta t)^\alpha}{\mathbb{AB}(\alpha)\Gamma(\alpha+1)} \sum_{m=2}^{n} t_{m-2}^{\gamma-1} \mathcal{F}_3(t_{m-2}, P^{m-2}) \times \left[ \Pi_k^n \right] \\
& + \frac{(\Delta t)^\alpha}{\mathbb{AB}(\alpha)\Gamma(\alpha+2)} \sum_{m=2}^{n} \left[ t_{m-1}^{\gamma-1} \mathcal{F}_3(t_{m-1}, P^{m-1}) - t_{m-2}^{\gamma-1} \mathcal{F}_3(t_{m-2}, P^{m-2}) \right] \times \left[ \Sigma_k^n \right] \\
& + \frac{(\Delta t)^\alpha}{2\mathbb{AB}(\alpha)\Gamma(\alpha+3)} \sum_{m=2}^{n} \left[ t_m^{\gamma-1} \mathcal{F}_3(t_m, P^m) - 2t_{m-1}^{\gamma-1} \mathcal{F}_3(t_{m-1}, P^{m-1}) + t_{m-2}^{\gamma-1} \mathcal{F}_3(t_{m-2}, P^{m-2}) \right] \times \left[ \Omega_k^n \right],
\end{aligned}
\end{equation}

\begin{equation}
\begin{aligned}
Z(t_{n+1}) = & \, \frac{1-\alpha}{\mathbb{AB}(\alpha)} t_n^{\gamma-1} \mathcal{F}_4(t_n, P^n) + \frac{(\Delta t)^\alpha}{\mathbb{AB}(\alpha)\Gamma(\alpha+1)} \sum_{m=2}^{n} t_{m-2}^{\gamma-1} \mathcal{F}_4(t_{m-2}, P^{m-2}) \times \left[ \Pi_k^n \right] \\
& + \frac{(\Delta t)^\alpha}{\mathbb{AB}(\alpha)\Gamma(\alpha+2)} \sum_{m=2}^{n} \left[ t_{m-1}^{\gamma-1} \mathcal{F}_4(t_{m-1}, P^{m-1}) - t_{m-2}^{\gamma-1} \mathcal{F}_4(t_{m-2}, P^{m-2}) \right] \times \left[ \Sigma_k^n \right] \\
& + \frac{(\Delta t)^\alpha}{2\mathbb{AB}(\alpha)\Gamma(\alpha+3)} \sum_{m=2}^{n} \left[ t_m^{\gamma-1} \mathcal{F}_4(t_m, P^m) - 2t_{m-1}^{\gamma-1} \mathcal{F}_4(t_{m-1}, P^{m-1}) + t_{m-2}^{\gamma-1} \mathcal{F}_4(t_{m-2}, P^{m-2}) \right] \times \left[ \Omega_k^n \right],
\end{aligned}
\end{equation}

\begin{equation}
\begin{aligned}
Z_a(t_{n+1}) = & \, \frac{1-\alpha}{\mathbb{AB}(\alpha)} t_n^{\gamma-1} \mathcal{F}_5(t_n, P^n) + \frac{(\Delta t)^\alpha}{\mathbb{AB}(\alpha)\Gamma(\alpha+1)} \sum_{m=2}^{n} t_{m-2}^{\gamma-1} \mathcal{F}_5(t_{m-2}, P^{m-2}) \times \left[ \Pi_k^n \right] \\
& + \frac{(\Delta t)^\alpha}{\mathbb{AB}(\alpha)\Gamma(\alpha+2)} \sum_{m=2}^{n} \left[ t_{m-1}^{\gamma-1} \mathcal{F}_5(t_{m-1}, P^{m-1}) - t_{m-2}^{\gamma-1} \mathcal{F}_5(t_{m-2}, P^{m-2}) \right] \times \left[ \Sigma_k^n \right] \\
& + \frac{(\Delta t)^\alpha}{2\mathbb{AB}(\alpha)\Gamma(\alpha+3)} \sum_{m=2}^{n} \left[ t_m^{\gamma-1} \mathcal{F}_5(t_m, P^m) - 2t_{m-1}^{\gamma-1} \mathcal{F}_5(t_{m-1}, P^{m-1}) + t_{m-2}^{\gamma-1} \mathcal{F}_5(t_{m-2}, P^{m-2}) \right] \times \left[ \Omega_k^n \right].
\end{aligned}
\end{equation}
The formulated loops construct a stable predictive mapping pattern for tracking complex biological transmission dynamics across memory-dependent fractional systems.

\begin{figure}[H]
    \centering
    \begin{subfigure}{0.48\textwidth}
        \centering
        \includegraphics[width=\linewidth]{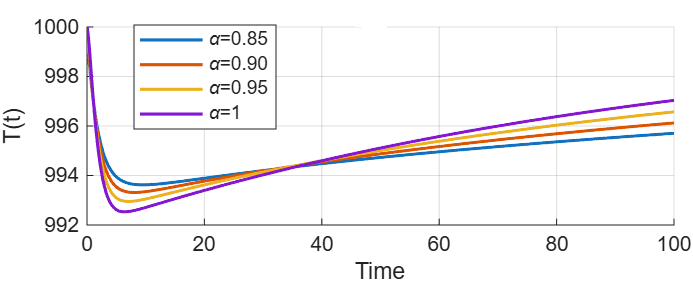}
        \caption{}
    \end{subfigure}
    \hfill 
    \begin{subfigure}{0.48\textwidth}
        \centering
        \includegraphics[width=\linewidth]{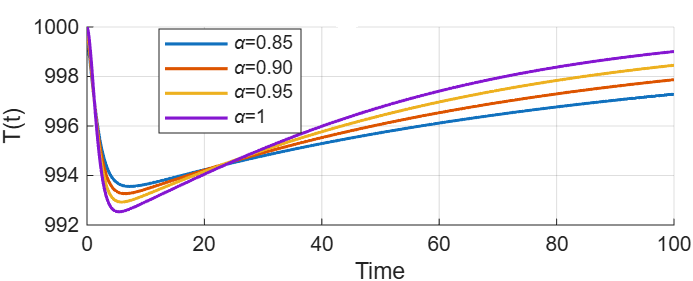}
        \caption{}
    \end{subfigure}
   \caption{Time series of T. (a) When $\gamma = 1$ and for various values of $\alpha$. (b) When $\gamma = 0.8$, and for various values of $\alpha$.} 
\label{fig:T}
\end{figure}
Figure (\ref{fig:T}) show time series of CD4$^{+}$ T-cells $T(\tilde{t})$, where $T(\tilde{t})$ initially decreases due to HIV infection through the interaction term $\chi T V$, and then gradually recovers as a result of continuous recruitment at rate $\lambda_T$ and natural death at rate $\mu_T$. For a fixed fractal dimension, larger values of $\alpha$ lead to faster dynamics and higher long-term levels of $T(\tilde{t})$. Moreover, reducing the fractal dimension from $p=1$ to $p=0.8$ enhances the memory effect, resulting in a slightly quicker recovery and increased susceptible CD4$^{+}$ T-cell levels. These results highlight the significant influence of both the fractional order and fractal dimension on the immune dynamics of the model.
\begin{figure}[H]
    \centering
    \begin{subfigure}{0.48\textwidth}
        \centering
        \includegraphics[width=\linewidth]{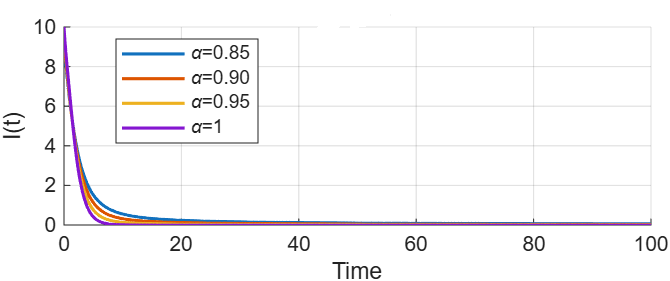}
        \caption{}
    \end{subfigure}
    \hfill 
    \begin{subfigure}{0.48\textwidth}
        \centering
        \includegraphics[width=\linewidth]{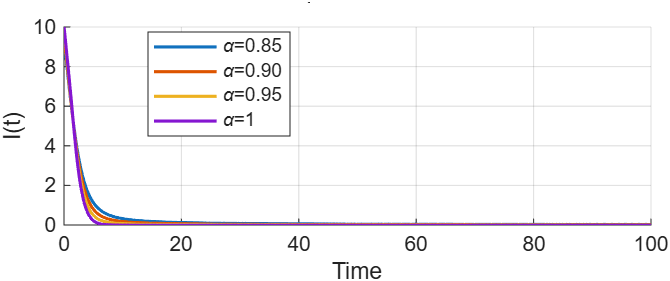}
        \caption{}
    \end{subfigure}
   \caption{Time series of I. (a) When $\gamma = 1$ and for various values of $\alpha$. (b) When $\gamma = 0.8$, and for various values of $\alpha$.} 
\label{fig:I}
\end{figure}
The figure (\ref{fig:I}) show the infected cell population decreases rapidly from its initial level due to the natural death rate $\mu_I$ and the elimination of infected cells by activated CD8$^{+}$ T-cells through the term $\alpha I Z_a$. For a fixed fractal dimension, increasing the fractional order $\alpha$ leads to a faster decay of $I(\tilde{t})$, indicating a quicker clearance of infected cells. Comparing both panels, the case $\gamma=0.8$ exhibits a slightly more pronounced memory effect, resulting in a marginally faster reduction of the infected population.\\
\begin{figure}[H]
    \centering
    \begin{subfigure}{0.48\textwidth}
        \centering
        \includegraphics[width=\linewidth]{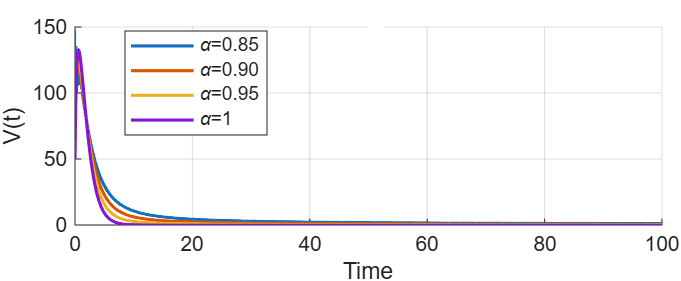}
        \caption{}
    \end{subfigure}
    \hfill 
    \begin{subfigure}{0.48\textwidth}
        \centering
        \includegraphics[width=\linewidth]{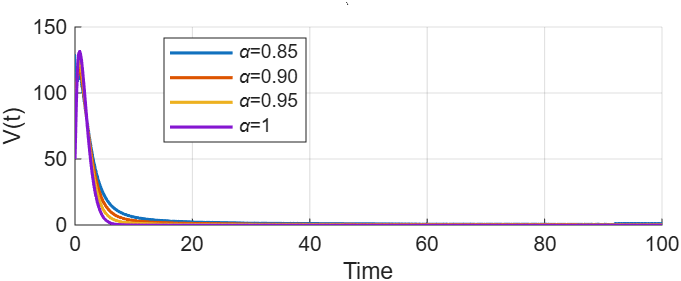}
        \caption{}
    \end{subfigure}
   \caption{Time series of V. (a) When $\gamma = 1$ and for various values of $\alpha$. (b) When $\gamma = 0.8$, and for various values of $\alpha$.} 
\label{fig:V}
\end{figure}
The figure (\ref{fig:V}) show the time evolution of the free HIV virus population $V(\tilde{t})$. The viral load decreases sharply at the early stage due to the high viral clearance rate $\mu_V$, and the decline of infected CD4$^{+}$ T-cells that produce new virions through the term $\varepsilon_V \mu_I I$. For a fixed fractal dimension, larger values of the fractional order $\alpha$ lead to a faster decay of $V(\tilde{t})$, indicating more rapid viral suppression. Comparing the two cases, the model with $\gamma = 0.8$ exhibits a slightly enhanced memory effect, resulting in a quicker reduction of the viral population than in the case $\gamma=1$.\\
Figure (\ref{fig:Z}) the dynamics of the CD8$^{+}$ T-cell population $Z(\tilde{t})$, In both cases, $Z(\tilde{t})$ decreases from its initial value due to natural death at rate $\mu_Z$ and activation into effector cells through the interaction term $\beta Z I$. The higher the value of the fractional order, the higher the rate at which the value of  $Z(\tilde{t})$ decays, indicating less significant memory effects and dynamics more akin to the classical integer-order model. With the fractal dimension decreased to $\gamma=0.8$, the decay of $Z(\tilde{t})$ is a little slower than when p=1, which shows that the fractal structure and the increased memory effects allow the population of $CD8^{+}$ T-cells to persist longer.\\

\begin{figure}[H]
    \centering
    \begin{subfigure}{0.48\textwidth}
        \centering
        \includegraphics[width=\linewidth]{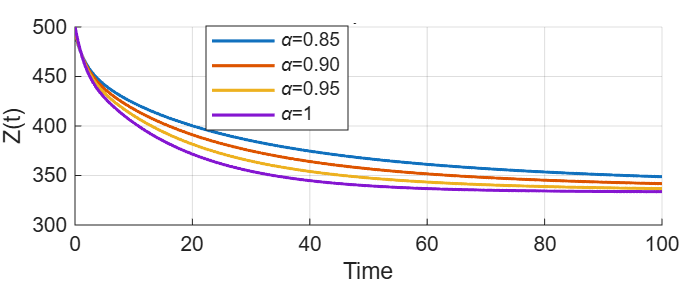}
        \caption{}
    \end{subfigure}
    \hfill 
    \begin{subfigure}{0.48\textwidth}
        \centering
        \includegraphics[width=\linewidth]{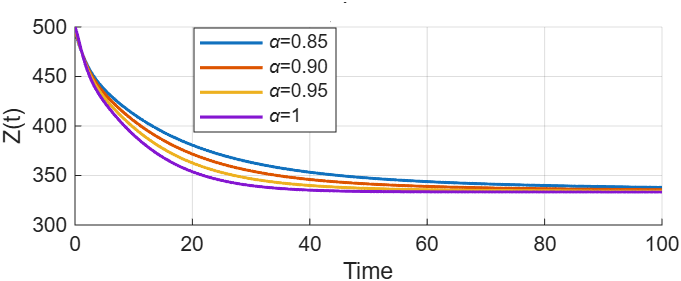}
        \caption{}
    \end{subfigure}
   \caption{Time series of Z. (a) When $\gamma = 1$ and for various values of $\alpha$. (b) When $\gamma = 0.8$, and for various values of $\alpha$.} 
\label{fig:Z}
\end{figure}
\begin{figure}[H]
    \centering
    \begin{subfigure}{0.48\textwidth}
        \centering
        \includegraphics[width=\linewidth]{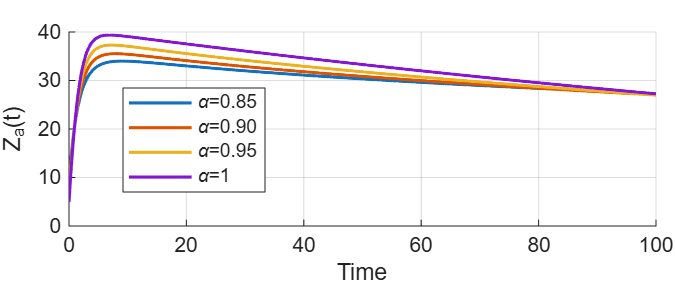}
        \caption{}
    \end{subfigure}
    \hfill 
    \begin{subfigure}{0.48\textwidth}
        \centering
        \includegraphics[width=\linewidth]{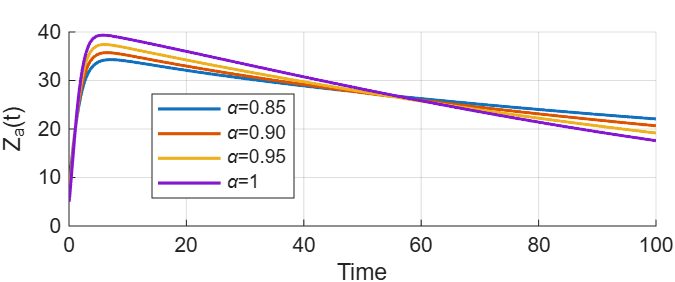}
        \caption{}
    \end{subfigure}
   \caption{Time series of $Z_a$. (a) When $\gamma = 1$ and for various values of $\alpha$. (b) When $\gamma = 0.8$, and for various values of $\alpha$.} 
\label{fig:Z_a}
\end{figure}
Figures (\ref{fig:Z_a} ) shows the time evolution of the activated $CD8^{+}$ T-cell population $Z_{a}(\tilde{t})$. Both scenarios have a steep rise of the $Z_a (\tilde{t})$ as a result of the activation process contained in the term of $beta Z I$ and a slow fall as a result of the natural decay rate of $Z a$ which is $mu Z a$. The larger the fractional orders $alpha$, the larger are the peak values and the faster the post-peak decay, and the less memory effects and the closer to the classical model. As the fractal dimension is reduced to $\gamma =0.8$, the memory action is increased, resulting in a smoother evolution and a more prolonged activated immune response than when $\gamma =1$. \\
\begin{figure}[h!]
    \centering
    \begin{subfigure}[b]{0.48\textwidth}
        \centering
        \includegraphics[width=\textwidth]{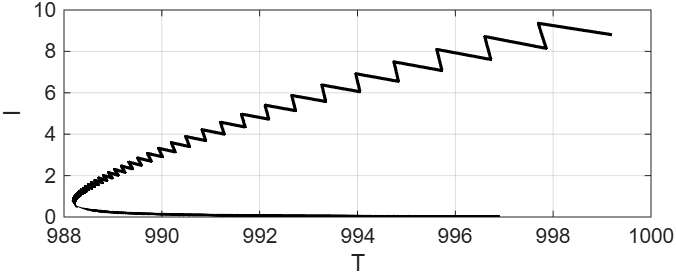}
        \caption{}
       \label{fig:ch1} 
    \end{subfigure}
    \hfill
    \begin{subfigure}[b]{0.48\textwidth}
        \centering
        \includegraphics[width=\textwidth]{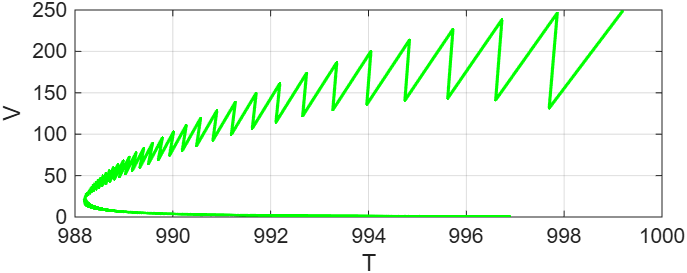}
        \caption{}
        \label{fig:ch2}
    \end{subfigure}
    \hfill
    \begin{subfigure}[b]{0.48\textwidth}
        \centering
        \includegraphics[width=\textwidth]{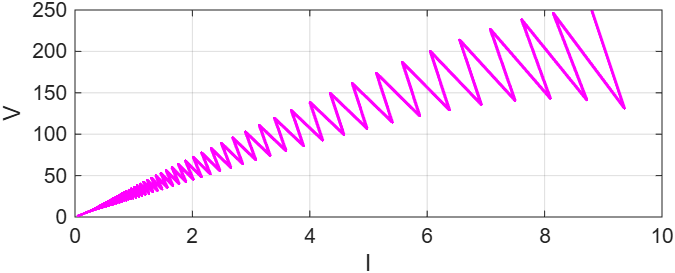}
        \caption{}
        \label{fig:ch3}
    \end{subfigure}
    \begin{subfigure}[b]{0.48\textwidth}
        \centering
        \includegraphics[width=\textwidth]{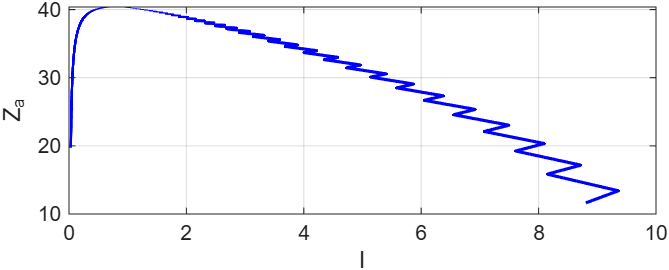}
        \caption{}
        \label{fig:ch5}
    \end{subfigure}
    \hfill
    \begin{subfigure}[b]{0.48\textwidth}
        \centering
        \includegraphics[width=\textwidth]{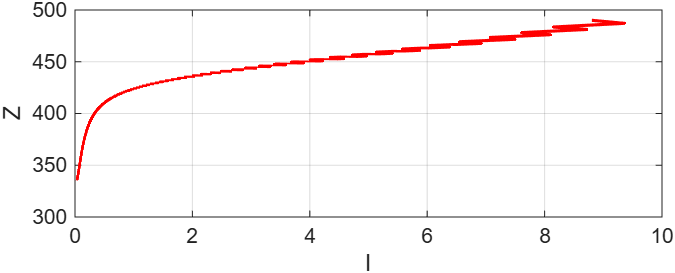}
        \caption{}
        \label{fig:ch6}
    \end{subfigure}
    \hfill
    \begin{subfigure}[b]{0.48\textwidth}
        \centering
        \includegraphics[width=\textwidth]{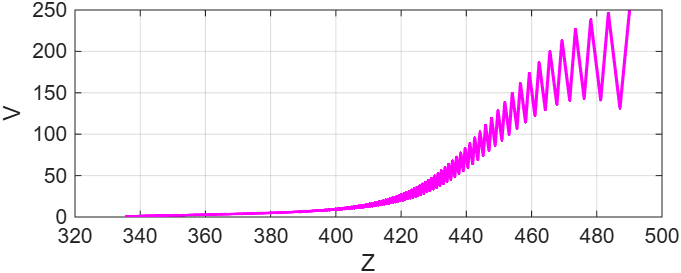}
        \caption{}
        \label{fig:ch7}
    \end{subfigure}
    \begin{subfigure}[b]{0.48\textwidth}
        \centering
        \includegraphics[width=\textwidth]{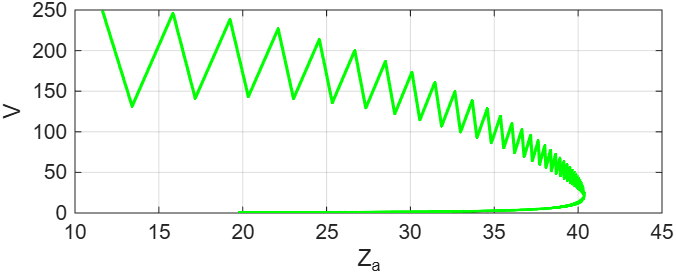}
        \caption{}
        \label{fig:ch8}
    \end{subfigure}
    \hfill
    \begin{subfigure}[b]{0.48\textwidth}
        \centering
        \includegraphics[width=\textwidth]{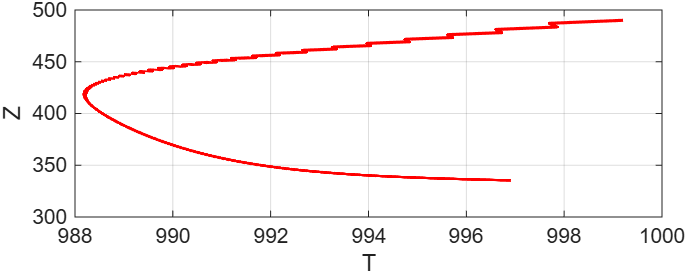}
        \caption{}
        \label{fig:ch9}
    \end{subfigure}
    \hfill
    \begin{subfigure}[b]{0.48\textwidth}
        \centering
        \includegraphics[width=\textwidth]{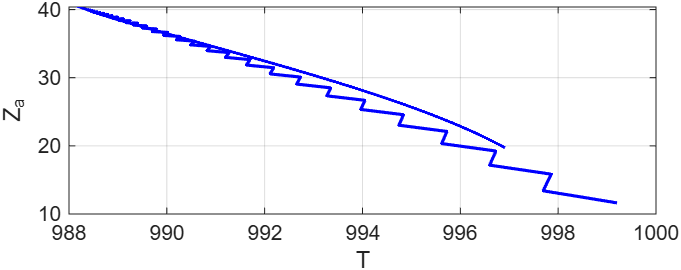}
        \caption{}
        \label{fig:ch10}
    \end{subfigure}
    
    \caption{Two-dimensional phase portraits of the fractal--fractional HIV model for $\alpha=0.9$, and $\gamma=0.9$ illustrating the nonlinear interactions among CD4$^{+}$ T-cells, infected cells, viral particles, CD8$^{+}$ T-cells, and activated immune cells.}
    \label{fig:nine_figs}
\end{figure}
\begin{figure}[H]
    \centering
    \begin{subfigure}[height = 15cm]{0.48\textwidth}
        \centering
        \includegraphics[width=\textwidth]{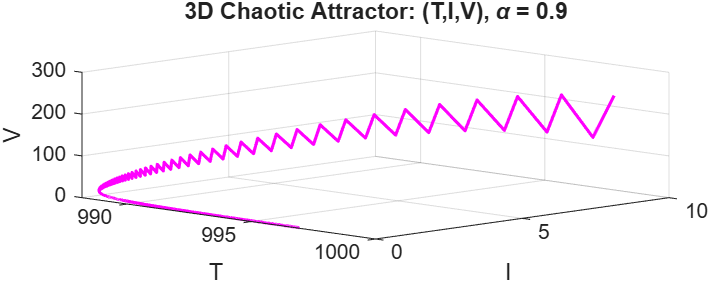}
        \caption{}
        \label{fig:3d1}
    \end{subfigure}
    \hfill
    \begin{subfigure}[height = 15cm]{0.48\textwidth}
        \centering
        \includegraphics[width=\textwidth]{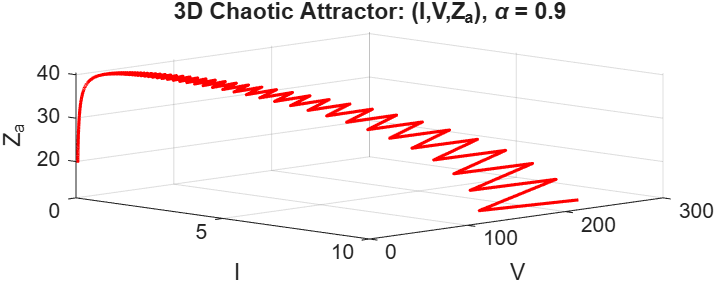}
        \caption{}
        \label{fig:3d2}
    \end{subfigure}
    \hspace{1cm}
    \hfill
    \begin{subfigure}[height = 15cm]{0.48\textwidth}
        \centering
        \includegraphics[width=\textwidth]{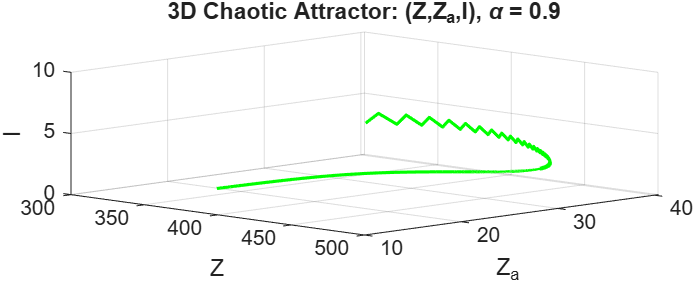}
        \caption{}
        \label{fig:3d3}
    \end{subfigure}
\caption{Three-dimensional chaotic attractors of the fractal--fractional HIV infection model for $\alpha=0.9$.}
\label{}
\end{figure}
\begin{figure}[h!]
    \centering
    \begin{subfigure}[height = 20cm]{0.48\textwidth}
        \centering
        \includegraphics[width=\textwidth]{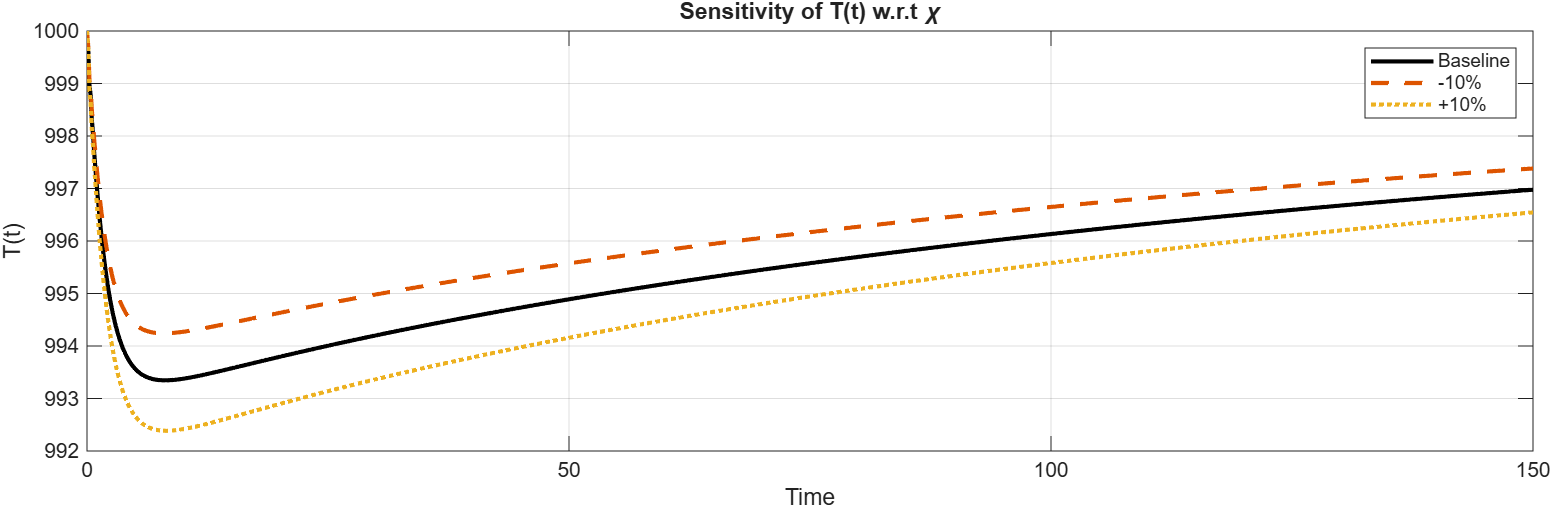}
        \caption{}
        \label{fig:03}
    \end{subfigure}
    \hfill
    \begin{subfigure}[height = 20cm]{0.48\textwidth}
        \centering
        \includegraphics[width=\textwidth]{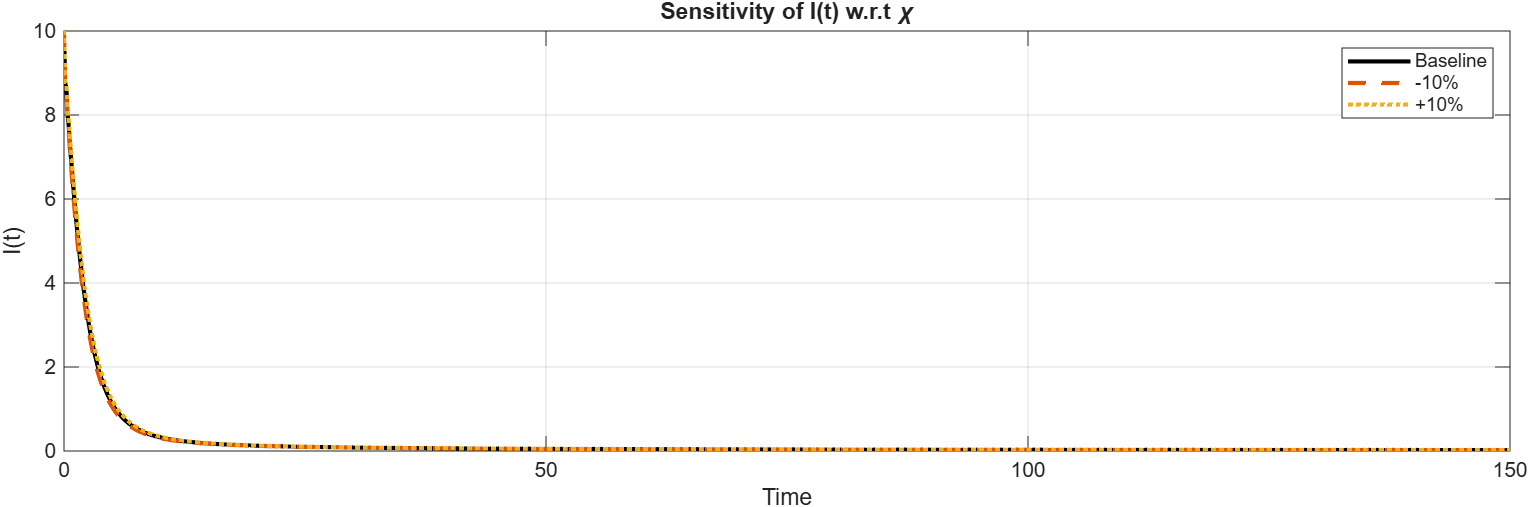}
        \caption{}
        \label{fig:04}
    \end{subfigure}
\begin{subfigure}[height = 20cm]{0.48\textwidth}
        \centering
        \includegraphics[width=\textwidth]{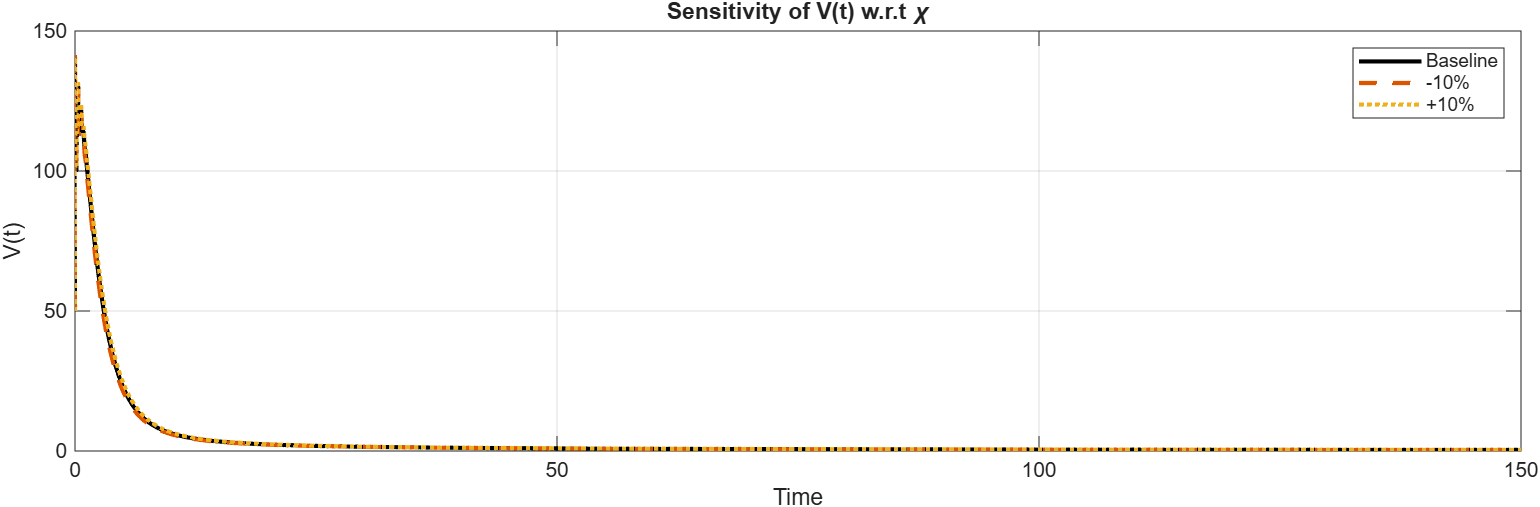}
        \caption{}
        \label{fig:05}
    \end{subfigure}
\begin{subfigure}[height = 20cm]{0.48\textwidth}
        \centering
        \includegraphics[width=\textwidth]{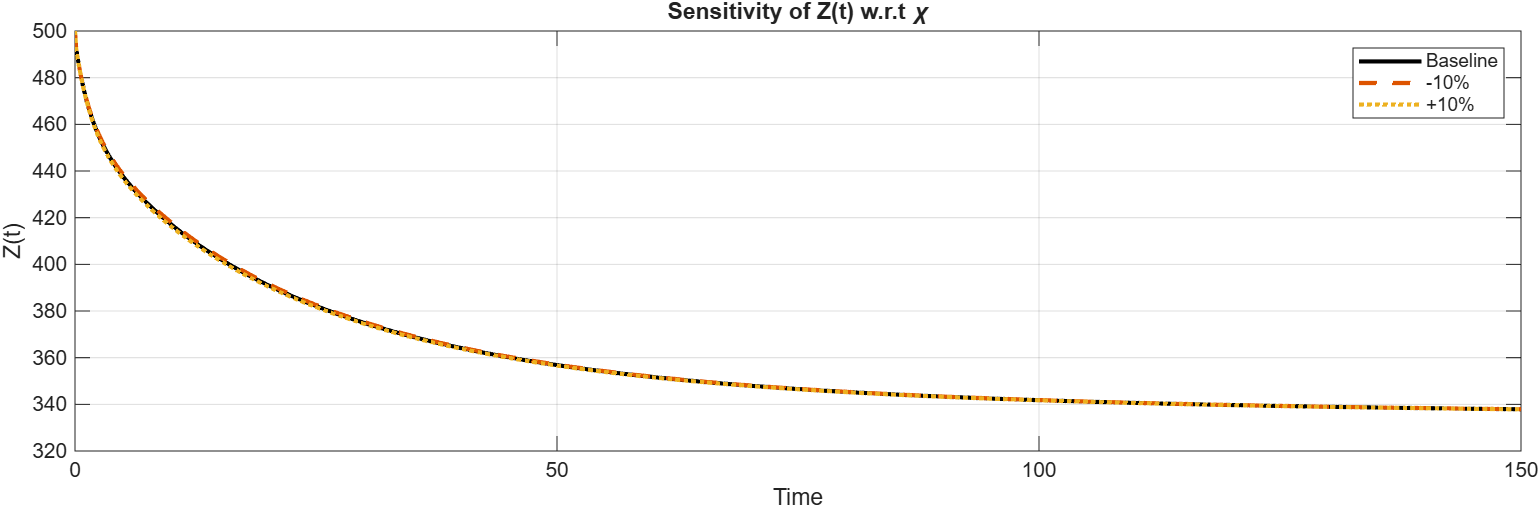}
        \caption{}
        \label{fig:06}
    \end{subfigure}
\begin{subfigure}[height = 20cm]{0.48\textwidth}
        \centering
        \includegraphics[width=\textwidth]{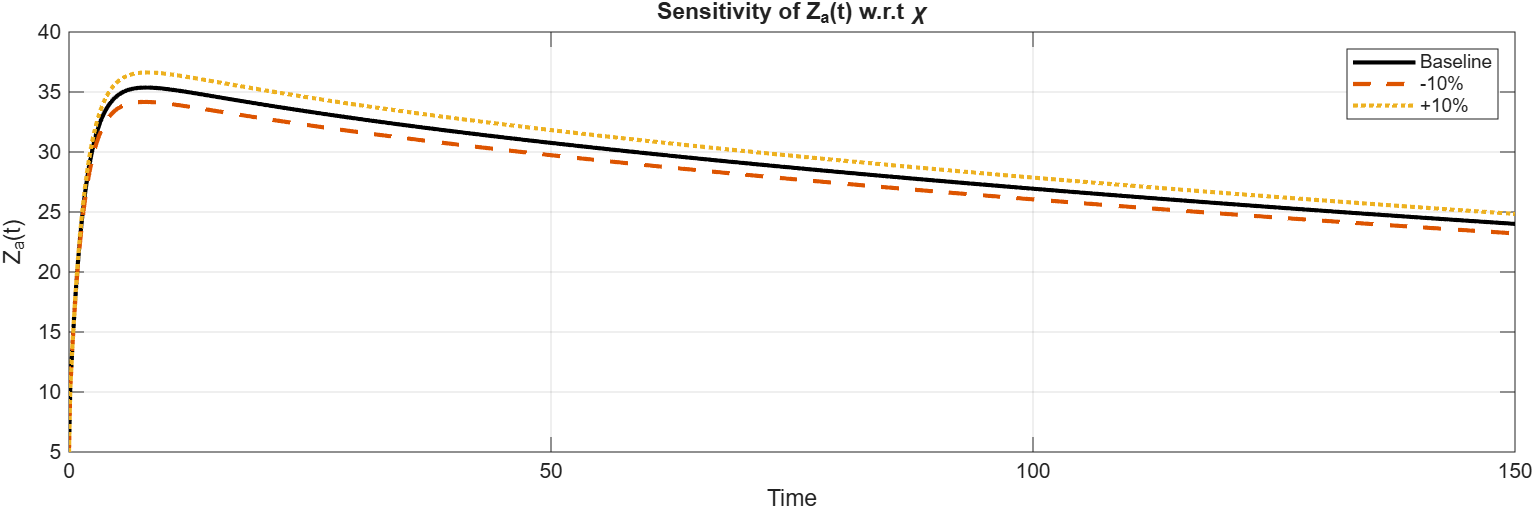}
        \caption{}
        \label{fig:07}
    \end{subfigure}
\caption{Trajectory-based sensitivity analysis of T-cells $T(\tilde{t})$ under $\pm 10\%$ perturbations of key parameters.}
\label{fig:9}
\end{figure}
\begin{figure}[h!]
    \centering
    \begin{subfigure}[height = 20cm]{0.48\textwidth}
        \centering
        \includegraphics[width=\textwidth]{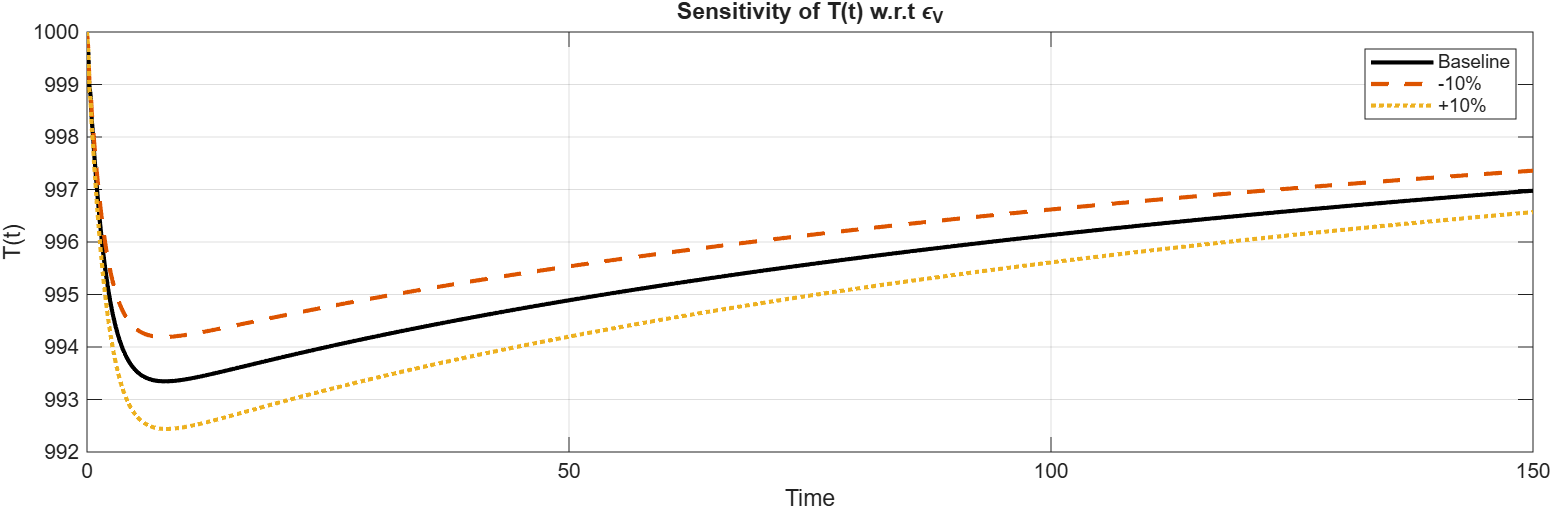}
        \caption{}
        \label{fig:08}
    \end{subfigure}
    \hfill
    \begin{subfigure}[height = 20cm]{0.48\textwidth}
        \centering
        \includegraphics[width=\textwidth]{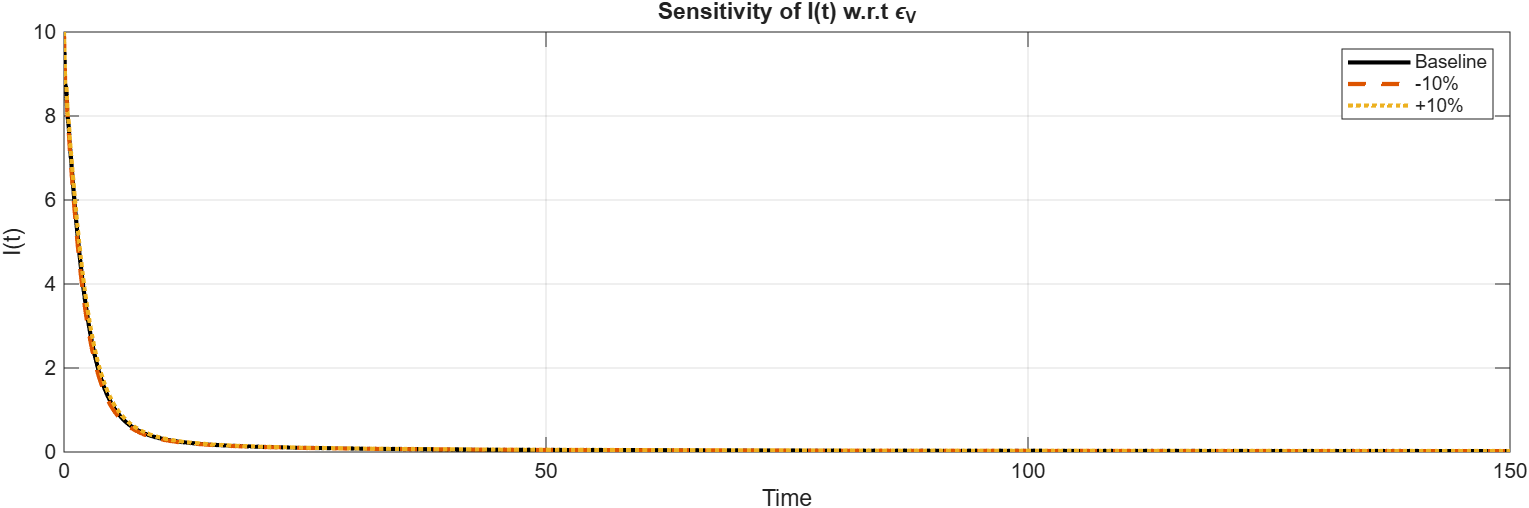}
        \caption{}
        \label{fig:09}
    \end{subfigure}
\begin{subfigure}[height = 20cm]{0.48\textwidth}
        \centering
        \includegraphics[width=\textwidth]{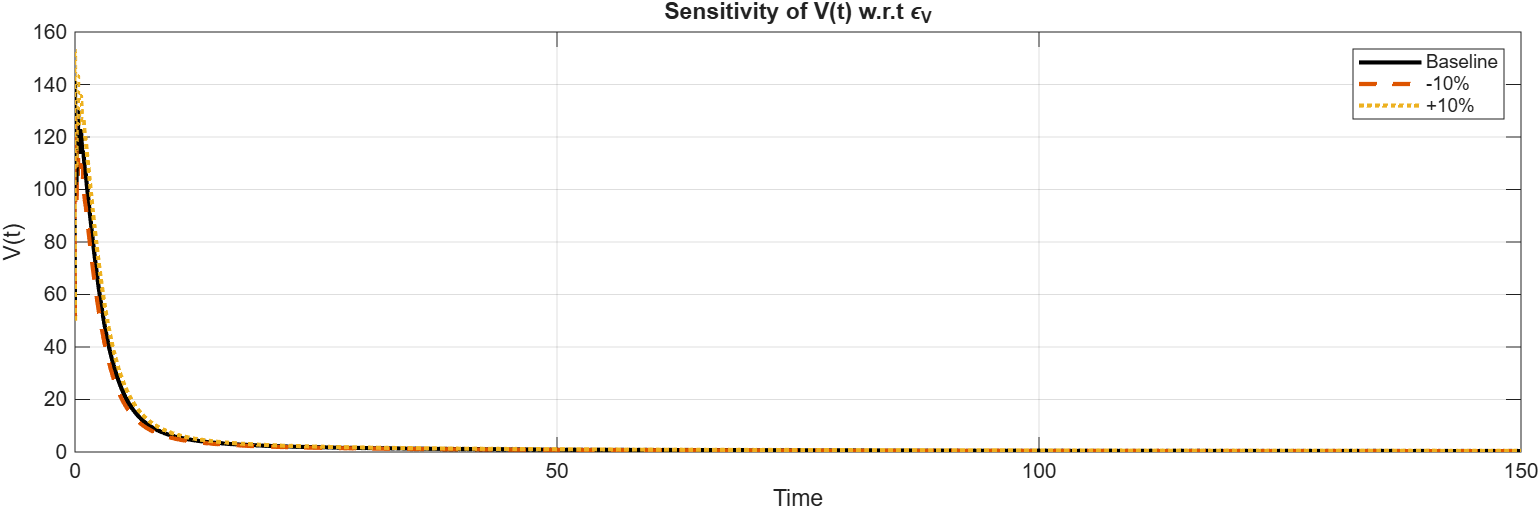}
        \caption{}
        \label{fig:010}
    \end{subfigure}
\begin{subfigure}[height = 20cm]{0.48\textwidth}
        \centering
        \includegraphics[width=\textwidth]{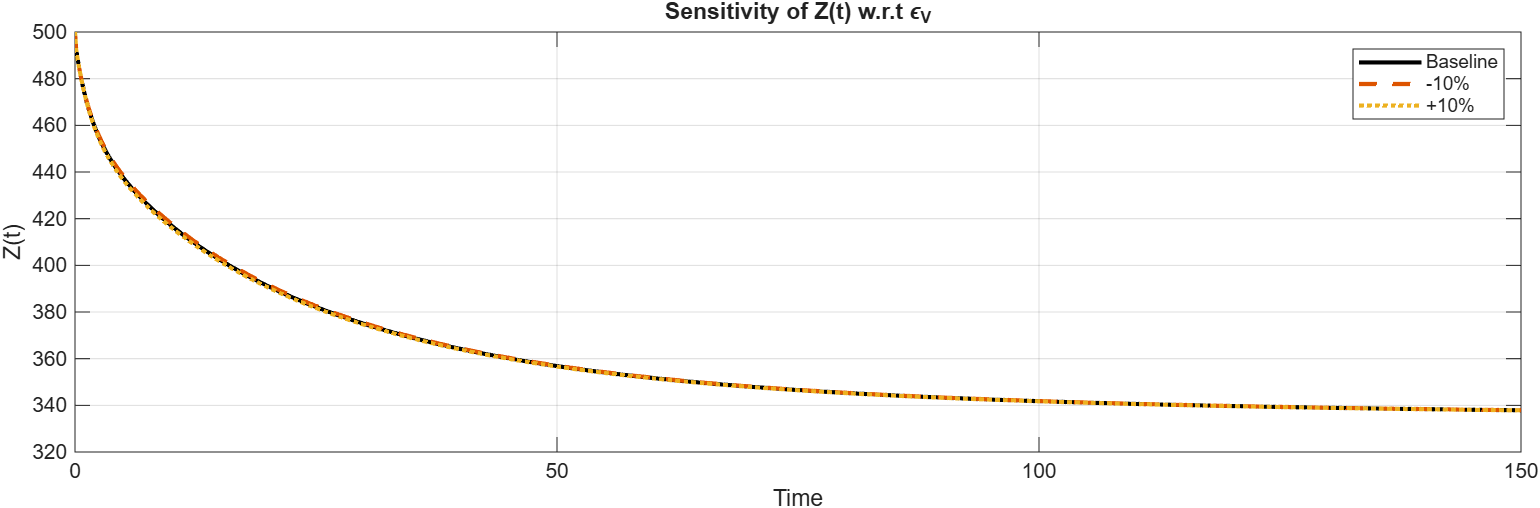}
        \caption{}
        \label{fig:011}
    \end{subfigure}
\begin{subfigure}[height = 20cm]{0.48\textwidth}
        \centering
        \includegraphics[width=\textwidth]{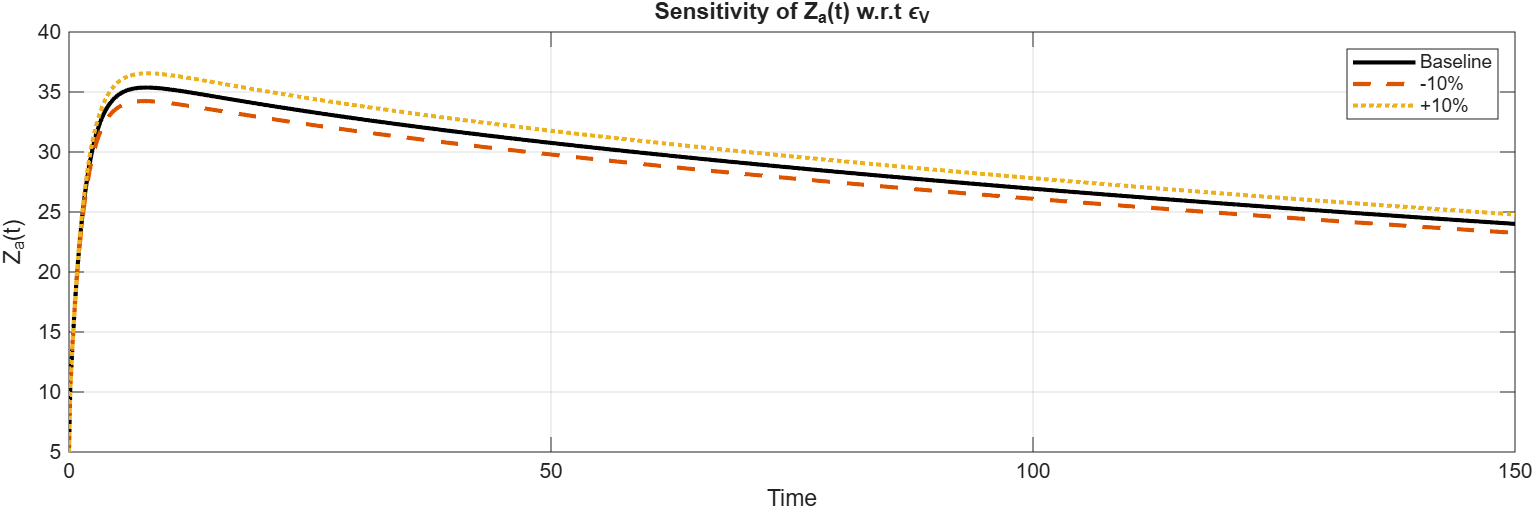}
        \caption{}
        \label{fig:012}
    \end{subfigure}
\caption{Trajectory-based sensitivity analysis of the infected CD4$^{+}$ T-cells $I(\tilde{t})$ under $\pm 10\%$ parameter variations.}
\label{fig:10}
\end{figure}
\begin{figure}[h!]
    \centering
    \begin{subfigure}[height = 20cm]{0.48\textwidth}
        \centering
        \includegraphics[width=\textwidth]{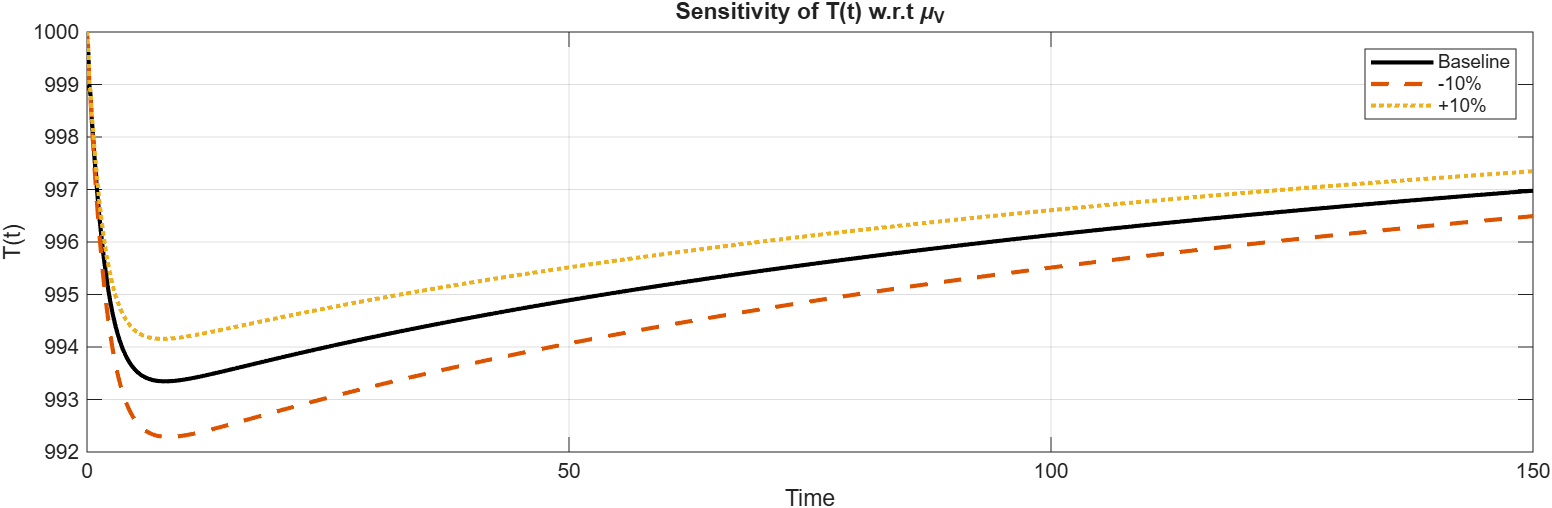}
        \caption{}
        \label{fig:013}
    \end{subfigure}
    \hfill
    \begin{subfigure}[height = 20cm]{0.48\textwidth}
        \centering
        \includegraphics[width=\textwidth]{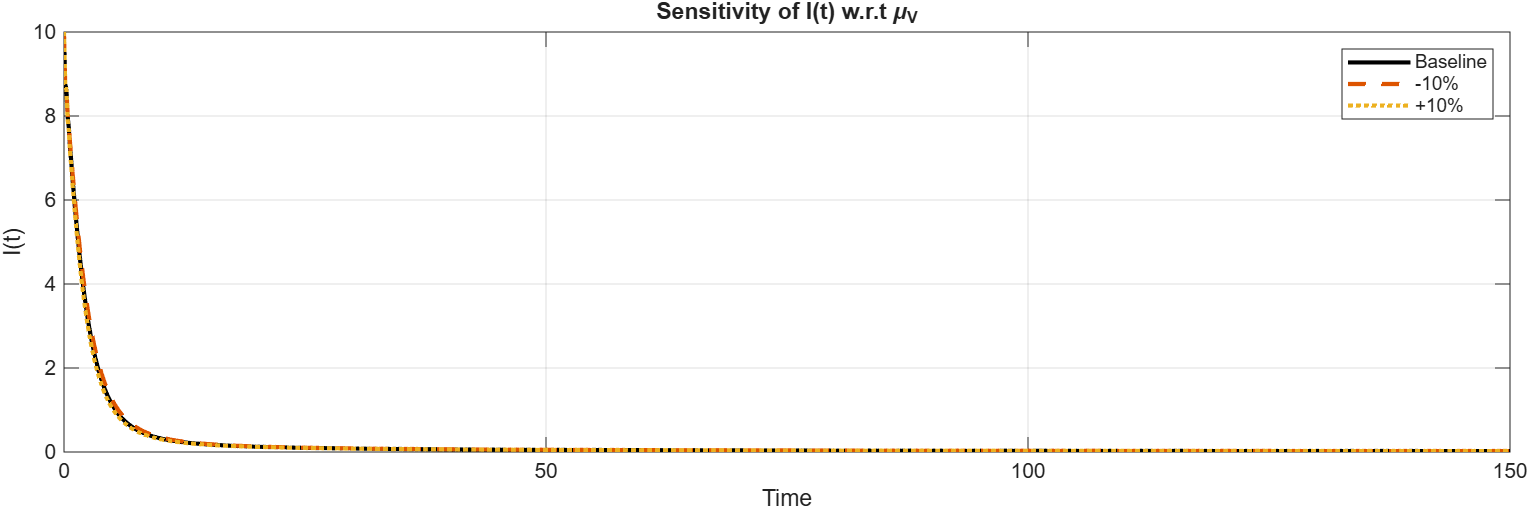}
        \caption{}
        \label{fig:014}
    \end{subfigure}
\begin{subfigure}[height = 20cm]{0.48\textwidth}
        \centering
        \includegraphics[width=\textwidth]{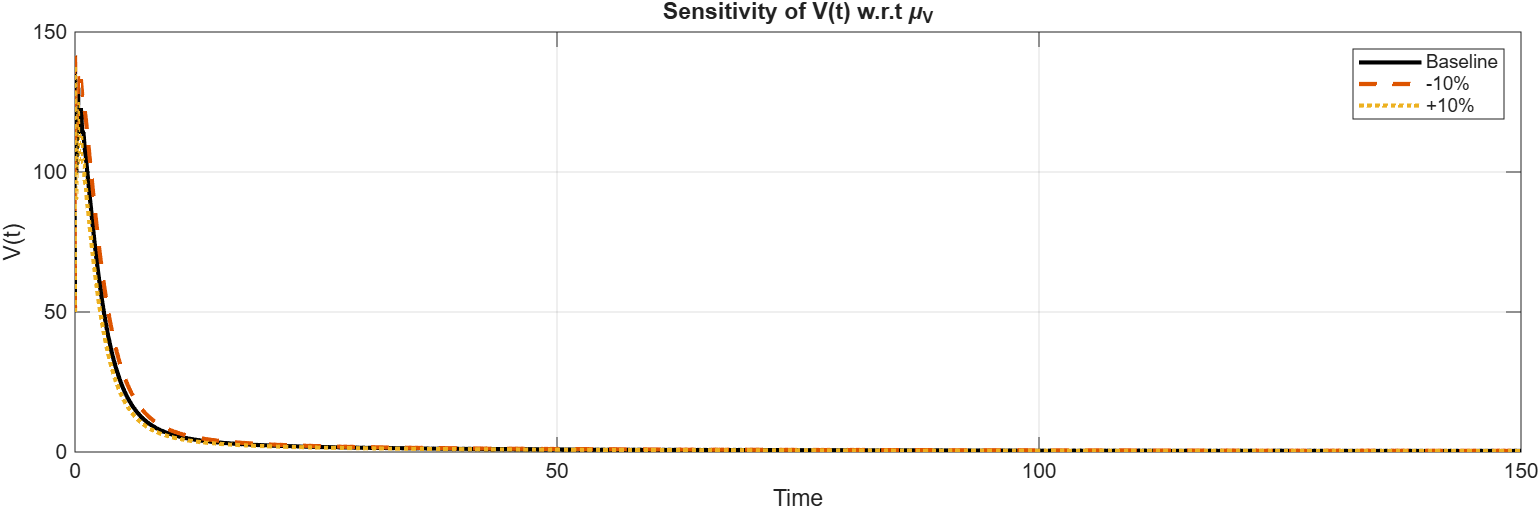}
        \caption{}
        \label{fig:015}
    \end{subfigure}
\begin{subfigure}[height = 20cm]{0.48\textwidth}
        \centering
        \includegraphics[width=\textwidth]{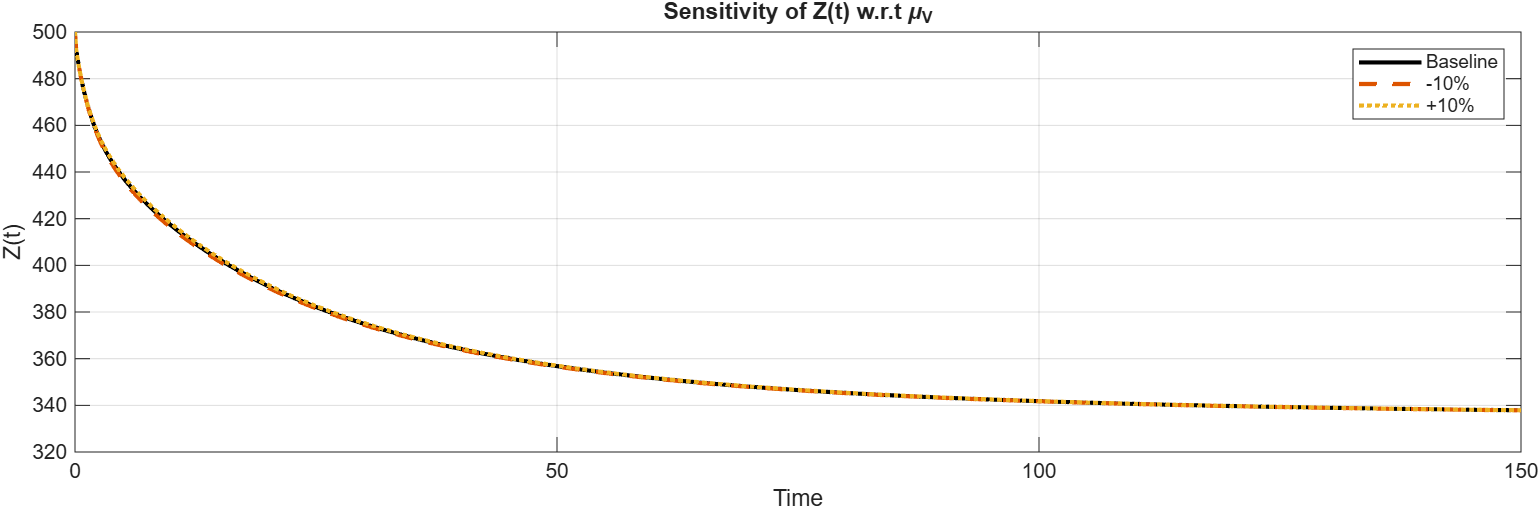}
        \caption{}
        \label{fig:016}
    \end{subfigure}
\begin{subfigure}[height = 20cm]{0.48\textwidth}
        \centering
        \includegraphics[width=\textwidth]{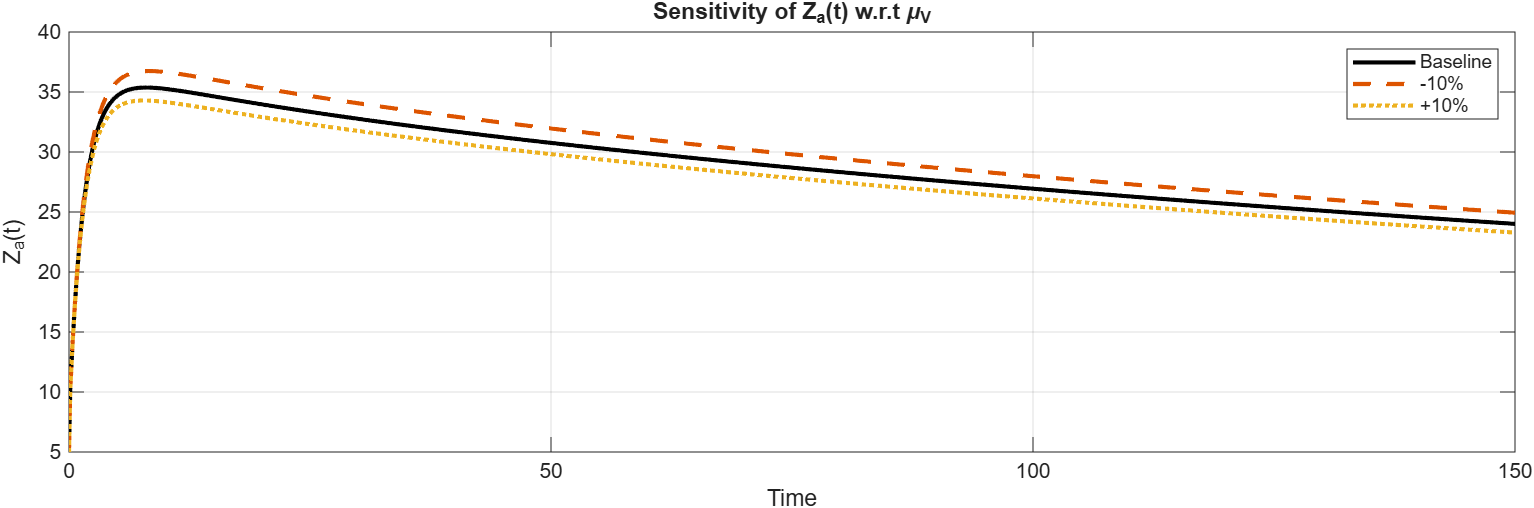}
        \caption{}
        \label{fig:017}
    \end{subfigure}
\caption{Trajectory-based sensitivity analysis of the viral load $V(\tilde{t})$ under $\pm 10\%$ perturbations in model parameters.}
\label{fig:11}
\end{figure}
\begin{figure}[H]
    \centering
    \begin{subfigure}[height = 20cm]{0.48\textwidth}
        \centering
        \includegraphics[width=\textwidth]{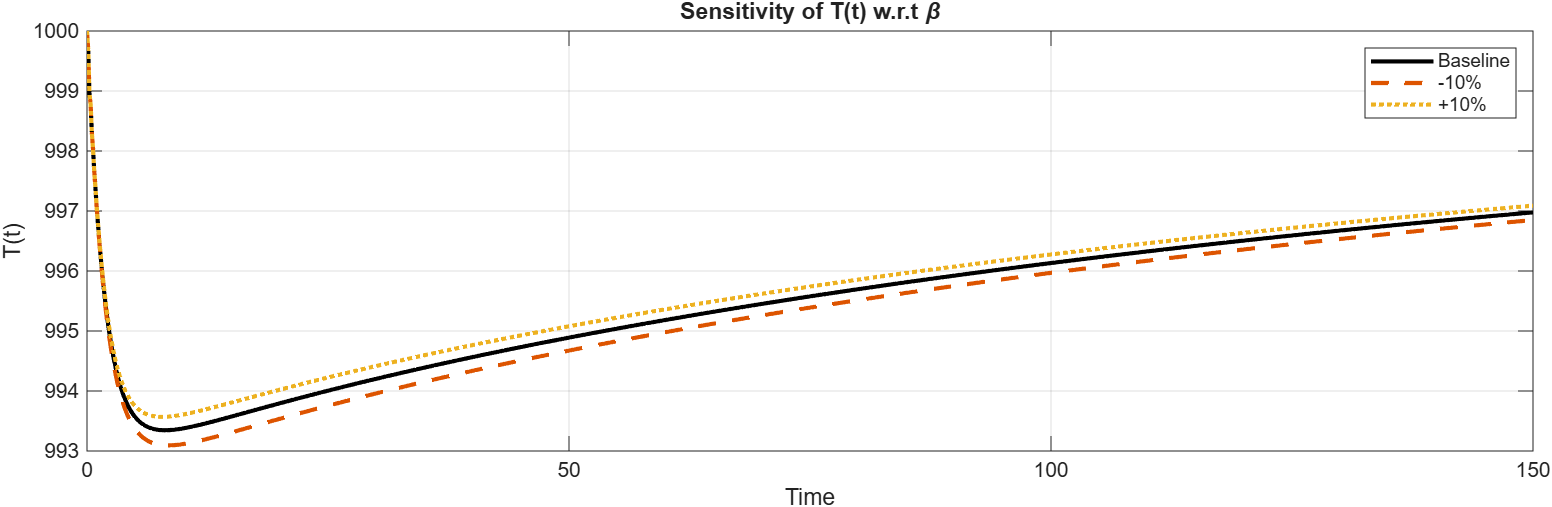}
        \caption{}
        \label{fig:018}
    \end{subfigure}
    \hfill
    \begin{subfigure}[height = 20cm]{0.48\textwidth}
        \centering
        \includegraphics[width=\textwidth]{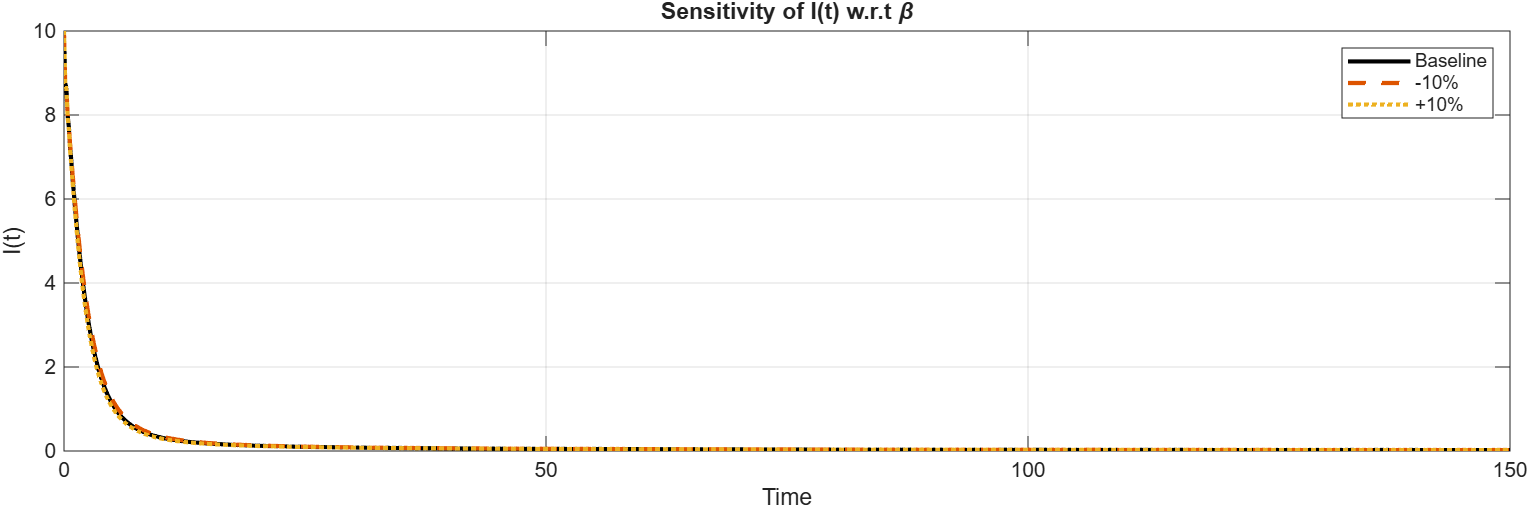}
        \caption{}
        \label{fig:019}
    \end{subfigure}
\begin{subfigure}[height = 20cm]{0.48\textwidth}
        \centering
        \includegraphics[width=\textwidth]{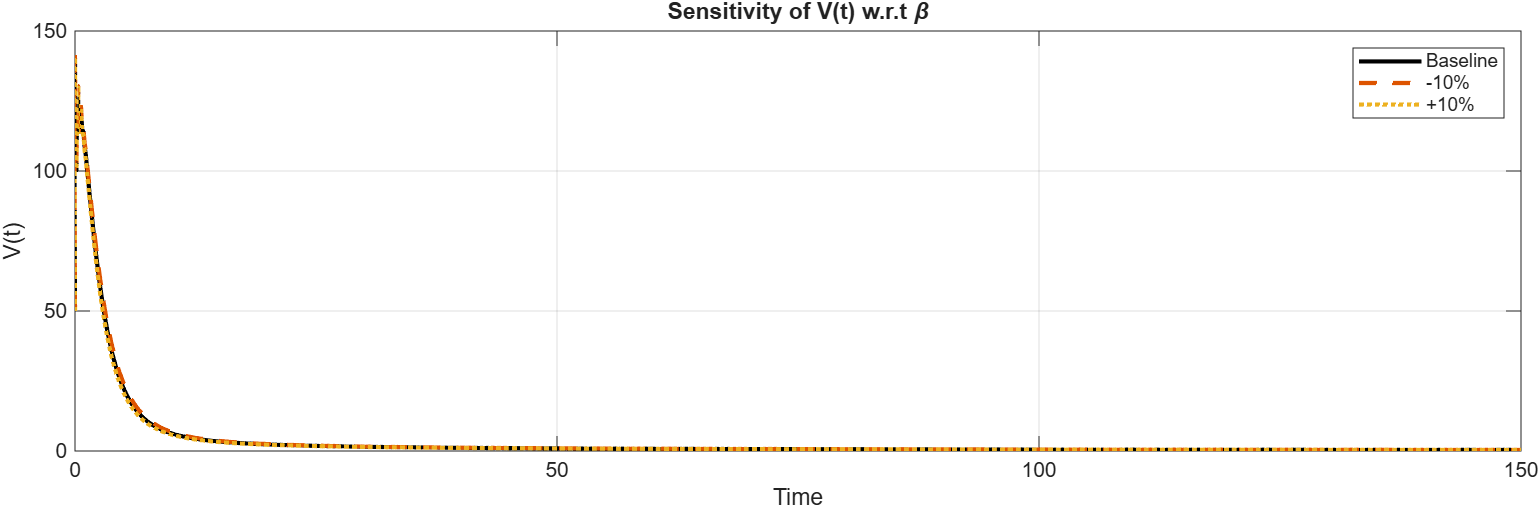}
        \caption{}
        \label{fig:020}
    \end{subfigure}
\begin{subfigure}[height = 20cm]{0.48\textwidth}
        \centering
        \includegraphics[width=\textwidth]{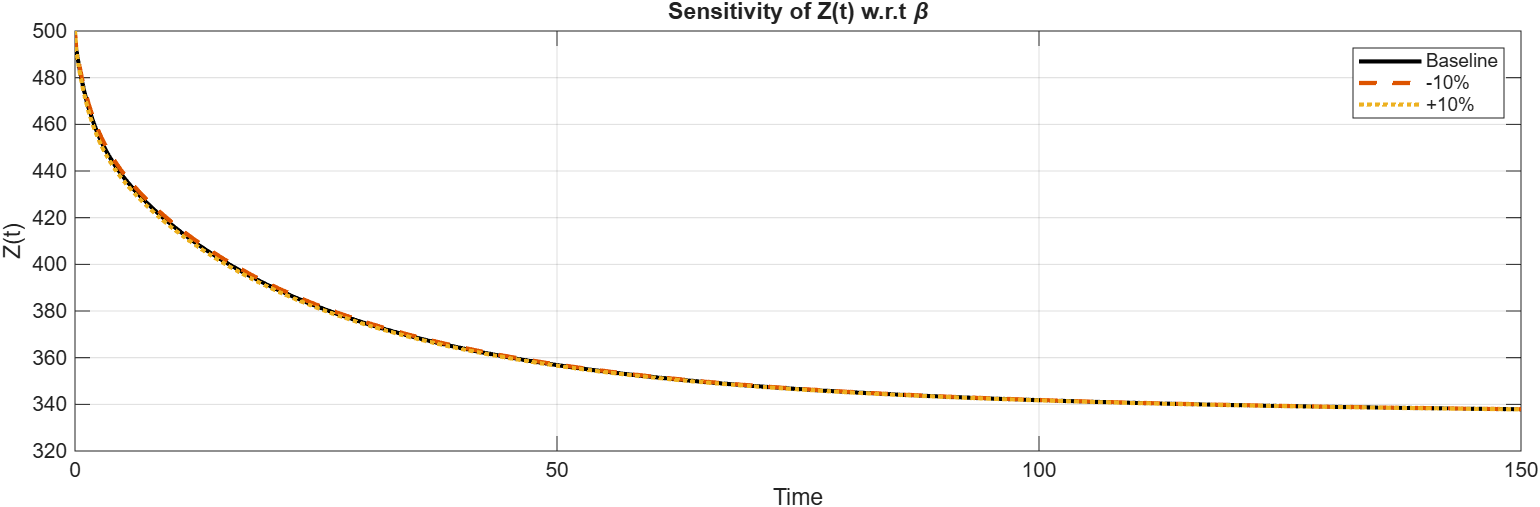}
        \caption{}
        \label{fig:021}
    \end{subfigure}
\begin{subfigure}[height = 20cm]{0.48\textwidth}
        \centering
        \includegraphics[width=\textwidth]{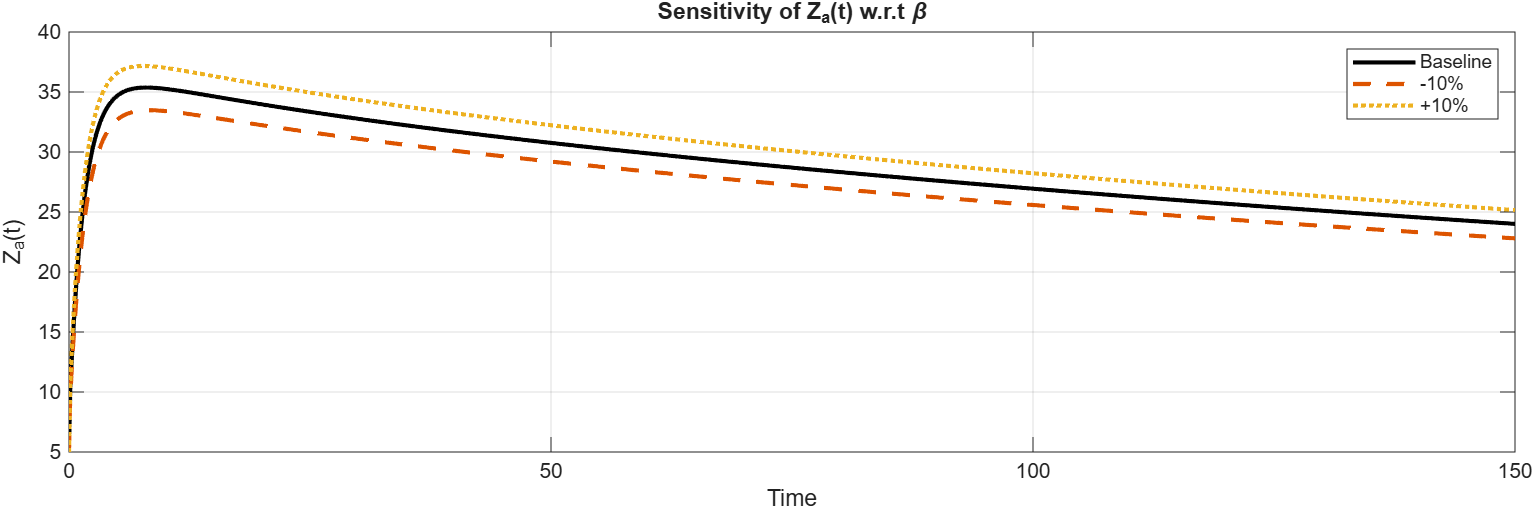}
        \caption{}
        \label{fig:022}
    \end{subfigure}
\caption{Trajectory-based sensitivity analysis of CD8$^{+}$ T-cells $Z(\tilde{t})$ under $\pm 10\%$ variations in key parameters.}
\label{fig:12}
\end{figure}
\begin{figure}[H]
    \centering
    \begin{subfigure}[height = 20cm]{0.48\textwidth}
        \centering
        \includegraphics[width=\textwidth]{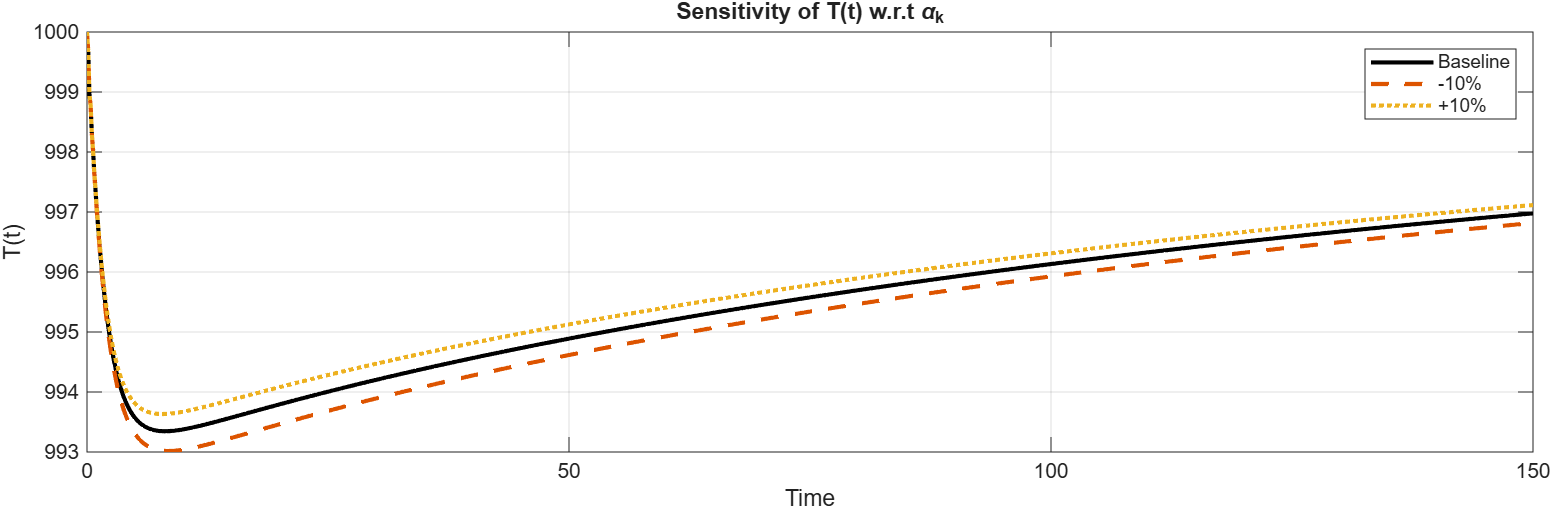}
        \caption{}
        \label{fig:023}
    \end{subfigure}
    \hfill
    \begin{subfigure}[height = 20cm]{0.48\textwidth}
        \centering
        \includegraphics[width=\textwidth]{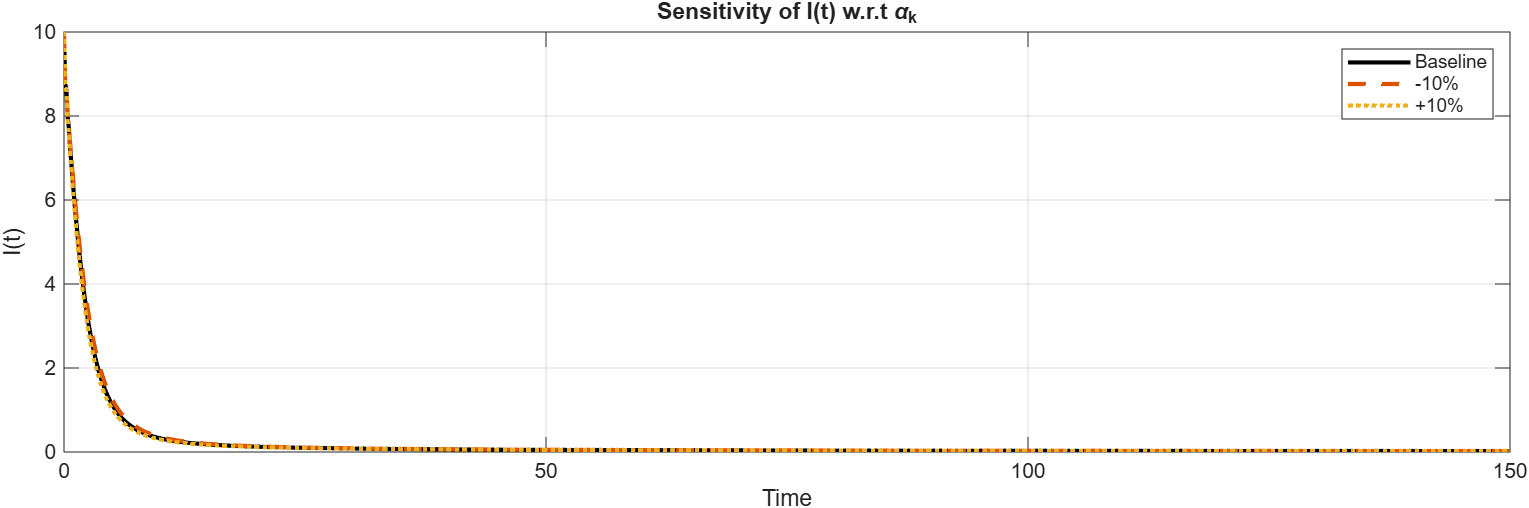}
        \caption{}
        \label{fig:024}
    \end{subfigure}
\begin{subfigure}[height = 20cm]{0.48\textwidth}
        \centering
        \includegraphics[width=\textwidth]{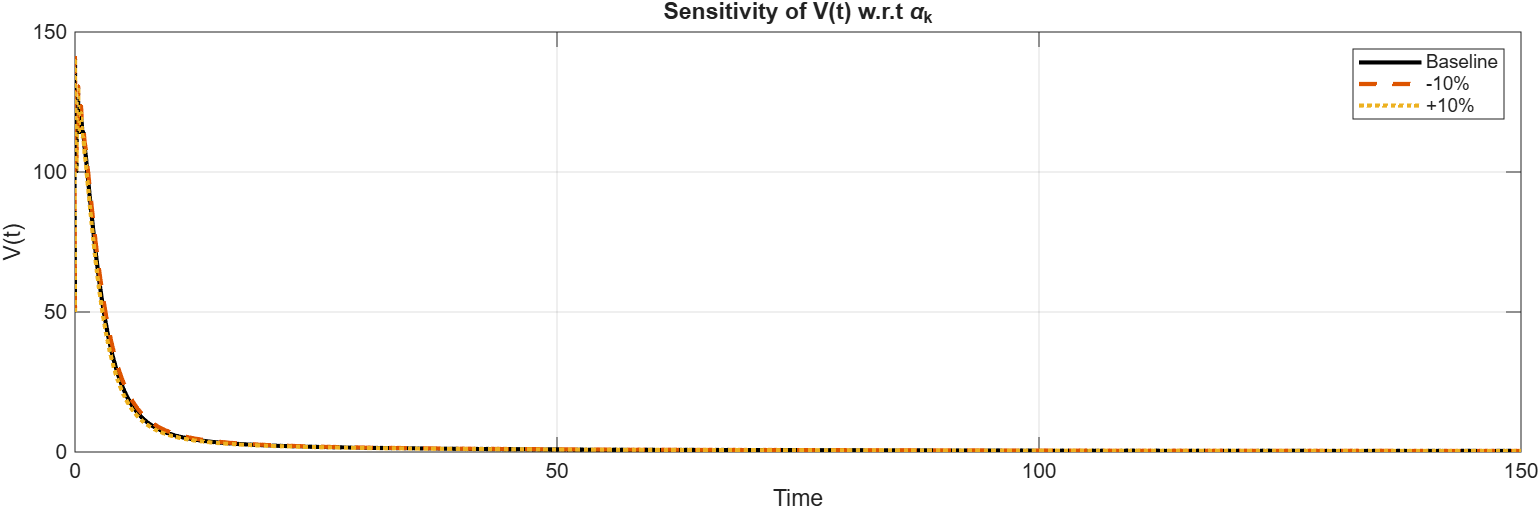}
        \caption{}
        \label{fig:025}
    \end{subfigure}
\begin{subfigure}[height = 20cm]{0.48\textwidth}
        \centering
        \includegraphics[width=\textwidth]{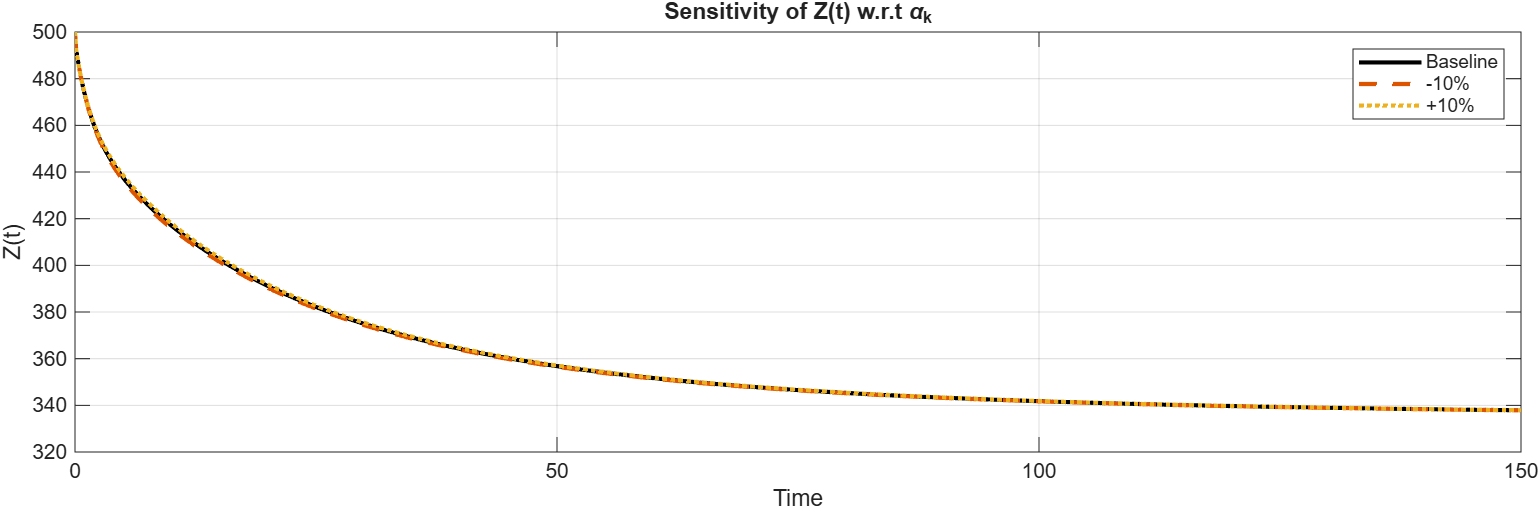}
        \caption{}
        \label{fig:026}
    \end{subfigure}
\begin{subfigure}[height = 20cm]{0.48\textwidth}
        \centering
        \includegraphics[width=\textwidth]{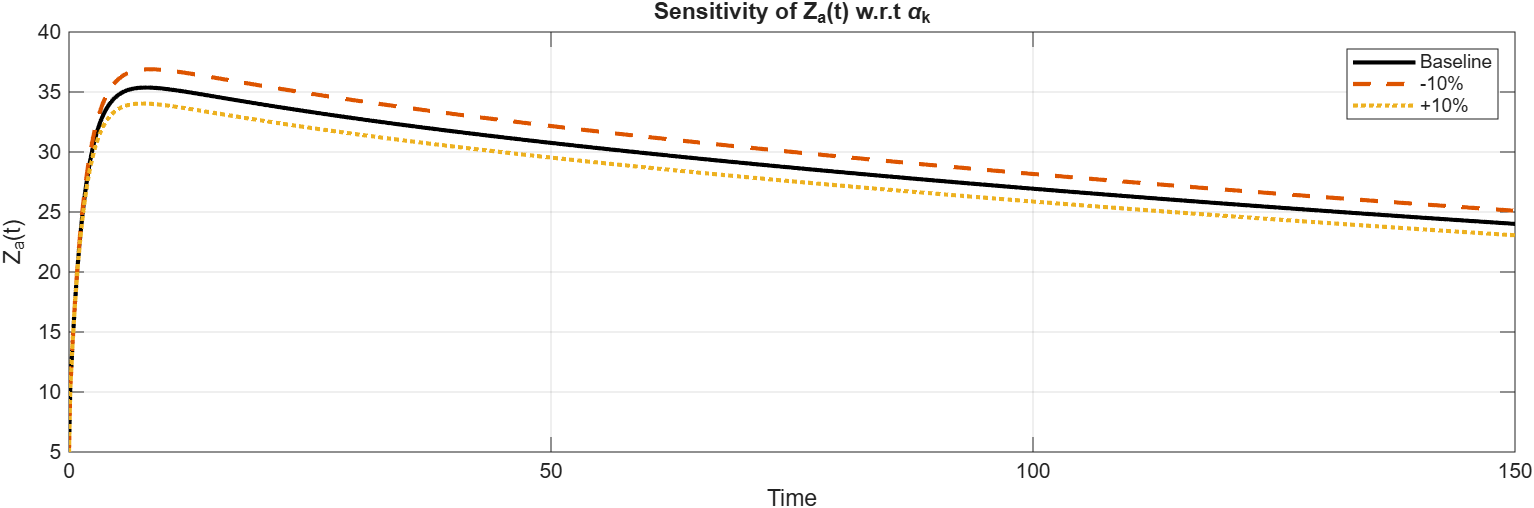}
        \caption{}
        \label{fig:027}
    \end{subfigure}
\caption{Trajectory-based sensitivity analysis of activated CD8$^{+}$ T-cells $Z_{a}(\tilde{t})$ under $\pm 10\%$ parameter perturbations.}
\label{fig:13}
\end{figure}
The given phase portraits (\ref{fig:ch1}--\ref{fig:ch10}) show the two dimensional projection of the fractal-fractional HIV model at $\alpha=0.9$. The curves in the $(T,I)$, $(T, V)$ and $(I, V)$ plots are oscillatory and sawtooth-like, showing that there are sustained nonlinear interactions among free virions, infected cells, and the susceptible CD4${+}$ T-cells. The components of the immune response $(Z,Z_a )$ also exhibit coupled oscillatory behavior, which indicates the activation mechanism based on the interaction terms, as is the case of the interaction terms, $\beta ZI$ and $\alpha_k I Z_a$. The predictions $(V,Z)$ and $(V,Z_a)$ indicate that changes in viral load are closely linked with the changes in the population of immune cells, which proves the feedback nature of the model. In general, the non-closed and oscillatory trajectories indicate the existence of a complex transient dynamics due to the memory effect of the fractional order and the nonlinear interaction among compartments, which indicates the rich dynamics of the proposed fractal-fractional HIV system (\ref{eq:FF_HIV_model}).\\
The three-dimensional phase portraits (\ref{fig:3d1}--\ref{fig:3d3}) depict the intertwined nonlinear behavior of the fractal fractional HIV model with $\alpha = 0.9$. The trajectory in the $(T,I,V)$ space has oscillatory and non-closed behaviors because of the high interplay between the free virions, infected cells, and the susceptible CD4${+}$ T-cells via the infection term $\chi T V$, and the viral production term $\varepsilon_V \mu_I I$. 
The projection of $(I,V,Z_a)$ brings out the feedback loop between the viral load, the activated T-cell response of the $CD8^{+}$, and the increased $V$ leads to the activation of the immune response indirectly through the interaction term $\beta$ $ZI$ and then regulated by the clearance term $\alpha_k$ $I Z_a$. 
In the same vein, the attractor of the immune dynamics on the space of $(Z,Za,I)$ illustrates the coordinated nature of the immune dynamics, and the way in which the processes of activation and natural decays influence the adaptive response. The lack of simple closed orbits and the existence of irregular oscillatory structures imply complex transient dynamics due to nonlinear coupling and memory effect brought about by the fractional order, $\alpha = 0.9$. These findings validate the rich dynamics of the proposed fractal  fractional HIV system.
\section{Visualization Methodology}
The presented plots (\ref{fig:9}--\ref{fig:13}) illustrate the trajectory-based sensitivity analysis of all compartments under parameter perturbations of $\pm 10\%$. The simulations were performed using the fractal--fractional HIV model with the Atangana--Baleanu--Caputo operator incorporating the Mittag--Leffler kernel, and solved numerically via the Atangana--Toufik scheme.
The results demonstrate that variations in key parameters significantly affect both transient and long-term dynamics of the system. In particular, the viral load $V(\tilde{t})$ and infected cells $I(\tilde{t})$ show strong sensitivity to changes in infection and production rates, while the immune compartments $Z(\tilde{t})$ and $Z_{a}(\tilde{t})$ exhibit noticeable responses to activation and clearance parameters. The susceptible CD4$^{+}$ T-cells $T(\tilde{t})$ display comparatively smoother variations, indicating a more stable behavior under perturbations.
Overall, the divergence of trajectories from the baseline solution confirms the strong nonlinear coupling and memory effects inherent in the model, highlighting the critical role of parameter values in shaping the progression and control of HIV dynamics.
\subsection{Sensitivity Heatmap Representation}
To visualize the influence of model parameters on system dynamics, a sensitivity heatmap is constructed. The normalized sensitivity indices of key output variables, particularly the viral load $V(\tilde{t})$, are computed with respect to selected parameters. These indices are evaluated over a range of fractional orders $\alpha$ and parameter values, and displayed as a color-coded matrix. This representation highlights regions of high and low sensitivity, providing a global view of parameter impact.
\begin{figure}[H]
\centering
\includegraphics[width=0.6\textwidth]{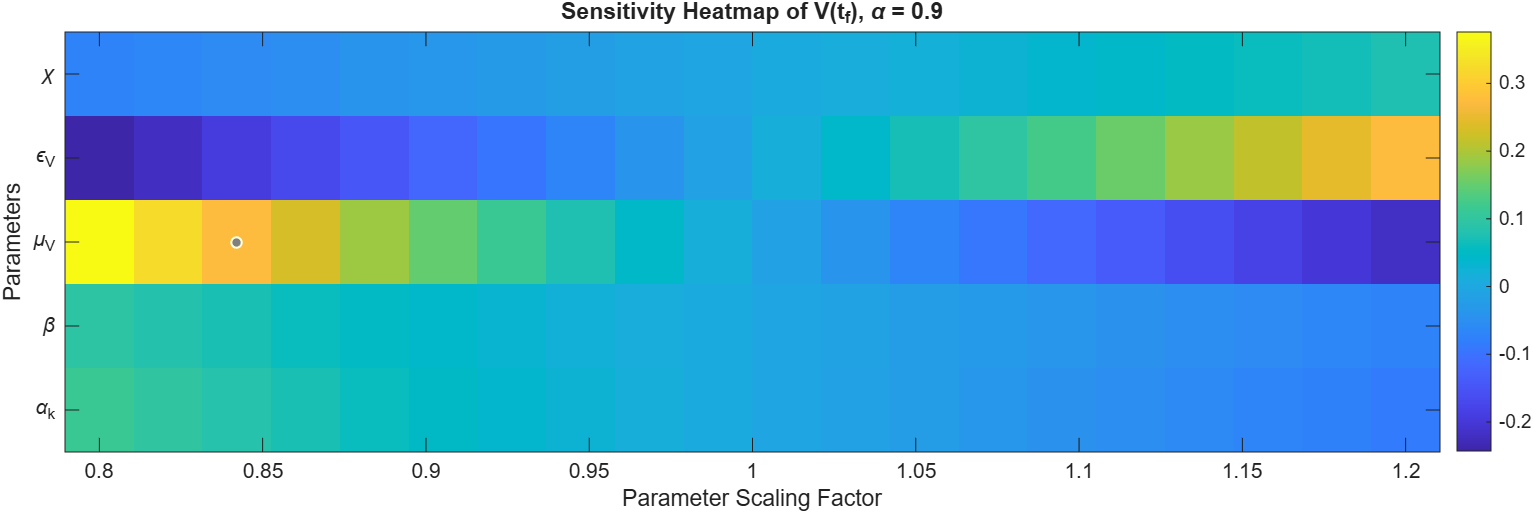}
\caption{Sensitivity heatmap showing the influence of key parameters on the viral load $V(t_f)$.}
\end{figure}

\subsection{Tornado Diagram Analysis}
A tornado diagram is employed to rank the relative importance of parameters affecting the model output. Each parameter is perturbed within a predefined range while keeping others fixed, and the resulting variation in the output (e.g., $V(t_f)$) is recorded. The parameters are then ordered according to their impact magnitude, allowing for clear identification of the most influential factors governing the system dynamics.
\begin{figure}[H]
\centering
\includegraphics[width=0.6\textwidth]{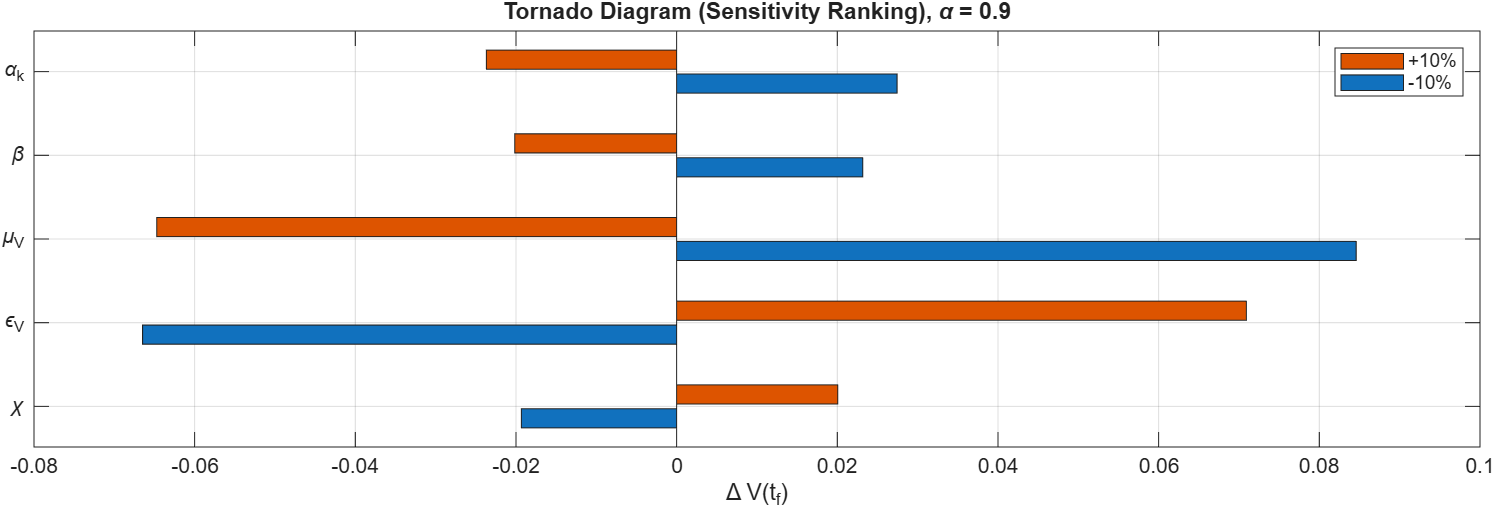}
\caption{Tornado diagram illustrating the relative importance of parameters on the viral load.}
\end{figure}

\section{Conclusion}
In this study, a fractal-fractional HIV model with the Mittag-Leffler kernel was developed using the Atangana--Baleanu--Caputo operator to describe the complex dynamics of HIV transmission. The theoretical analysis confirmed the existence and uniqueness of the solutions, while the Hyers--Ulam stability analysis verified the stability and reliability of the proposed model. Numerical simulations were successfully performed using the Newton polynomial approximation together with the Atangana--Toufik numerical scheme. The obtained results demonstrated that the considered parameters have significant effects on both transient and long-term behaviors of the HIV system. Moreover, visualization tools such as sensitivity heatmaps and tornado diagrams provided deeper insight into the influence of model parameters on disease dynamics. Overall, the proposed fractal--fractional framework offers an effective mathematical and computational approach for studying HIV dynamics and can be extended to other complex biological systems such as Malaria and COVID infections in future investigations.

\section*{Declaration}

\noindent \textbf{Conflict of Interest:} The authors declared that they have no conflict of interest.

\noindent \textbf{Funding:}
The work of S. Noeiaghdam was funded by the High-Level Talent Research Start-up Project Funding of Henan Academy of Sciences (Project No. 241819246). This research was supported by Henan International Studio for Stochastic Mathematical Modeling and Artificial Intelligence in Energy and Battery Systems, No. GZS2026018. 

\noindent \textbf{Author Contributions:} N.S. and S.N. wrote the main manuscript text, edited, software and coding and prepared the results. S.N. supervised the project. All authors reviewed the manuscript.

\noindent \textbf{Data Availability:} Not applicable. 

\noindent \textbf{AI Use:} We declare that we applied AI-tools for English grammar check.

\end{document}